\documentclass[reqno,11pt]{amsart}
\usepackage{amsmath, latexsym, amsfonts, amssymb, amsthm, amscd,mathrsfs}
\usepackage{graphics,epsf,psfrag}
\usepackage[demo]{graphicx}
\usepackage{caption}
\usepackage{subcaption}
\numberwithin{equation}{section}

\newtheorem{maintheorem}{Theorem}
\newtheorem{theorem}{Theorem}[section]
\newtheorem{fact}[theorem]{Fact}

\newtheorem{lemma}[theorem]{Lemma}
\newtheorem{proposition}[theorem]{Proposition}

\newtheorem{remark}[theorem]{Remark}
\newtheorem{definition}[theorem]{Definition}

\newtheorem{example}[theorem]{Example}
\newtheorem{prob}[theorem]{Problem}
\newtheorem{question}[theorem]{Question}

\newcommand{\ind}{\mathbf{1}}

\newcommand{\R}{\mathbb{R}}
\newcommand{\Z}{\mathbb{Z}}
\newcommand{\N}{\mathbb{N}}
\renewcommand{\tilde}{\widetilde}
\renewcommand{\hat}{\widehat}

\newcommand{\cT}{{\ensuremath{\mathcal T}} }
\newcommand{\cD}{{\ensuremath{\mathcal D}} }

\DeclareMathSymbol{\leqslant}{\mathalpha}{AMSa}{"36} 
\DeclareMathSymbol{\geqslant}{\mathalpha}{AMSa}{"3E} 
\DeclareMathSymbol{\eset}{\mathalpha}{AMSb}{"3F}     
\renewcommand{\leq}{\;\leqslant\;}                   
\renewcommand{\geq}{\;\geqslant\;}                   

\DeclareMathOperator*{\Var}{\mathrm Var}      

\newcommand{\pin}{\ensuremath{\pi_n}}

\newcommand{\gep}{\varepsilon}       

\newcommand{\gO}{\Omega}
\newcommand{\gl}{\lambda}

\makeatletter
\def\captionfont@{\footnotesize}
\def\captionheadfont@{\scshape}

\long\def\@makecaption#1#2{%
  \vspace{2mm}
  \setbox\@tempboxa\vbox{\color@setgroup
    \advance\hsize-6pc\noindent
    \captionfont@\captionheadfont@#1\@xp\@ifnotempty\@xp
        {\@cdr#2\@nil}{.\captionfont@\upshape\enspace#2}%
    \unskip\kern-6pc\par
    \global\setbox\@ne\lastbox\color@endgroup}%
  \ifhbox\@ne 
    \setbox\@ne\hbox{\unhbox\@ne\unskip\unskip\unpenalty\unkern}%
  \fi
  \ifdim\wd\@tempboxa=\z@ 
    \setbox\@ne\hbox to\columnwidth{\hss\kern-6pc\box\@ne\hss}%
  \else 
    \setbox\@ne\vbox{\unvbox\@tempboxa\parskip\z@skip
        \noindent\unhbox\@ne\advance\hsize-6pc\par}%
\fi
  \ifnum\@tempcnta<64 
    \addvspace\abovecaptionskip
    \moveright 3pc\box\@ne
  \else 
    \moveright 3pc\box\@ne
    \nobreak
    \vskip\belowcaptionskip
  \fi
\relax
}
\makeatother
\def\writefig#1 #2 #3 {\rlap{\kern #1 truecm
\raise #2 truecm \hbox{#3}}}

\newcommand{\Cov}{{\bf Cov}}

\renewcommand{\Pr}{ \mathrm P}
\newcommand{ \rel}{ t_{\mathrm{rel}} }
\newcommand{ \reln}{ t_{\mathrm{rel}}^{(n)} }
\newcommand{ \mix}{ t_{\mathrm{mix}} }
\newcommand{ \mixn}{ t_{\mathrm{mix}}^{(n)} }
\newcommand{ \dsep}{ d_{\mathrm{sep}} }
\newcommand{ \dsepn}{ d_{\mathrm{sep}}^{(n)} }

\newcommand{ \sepn}{ t_{\mathrm{sep}}^{(n)} }
\newcommand{ \sepeps}{ t_{\mathrm{sep}}(\epsilon) }

\newcommand{ \mixneps}{ t_{\mathrm{mix}}^{(n)}(\epsilon) }

\newcommand{ \h}{ \mathrm{H} }
\newcommand{ \TV}{ \mathrm{TV} }

\newcommand{ \cL}{ \mathcal L }

\newcommand{\eps}{\epsilon }
\newcommand{\del}{\delta }

\newcommand{\tf}{\textsc{f}}

\usepackage[normalem]{ulem}

\begin{document}

\title{On sensitivity of mixing times and cutoff}
\author{
Jonathan Hermon
\and
Yuval Peres}

\thanks{
University of Cambridge, Cambridge, UK. E-mail: {\tt jonathan.hermon@statslab.cam.ac.uk}. Financial support by
the EPSRC grant EP/L018896/1.}

\thanks{
Microsoft Research, Redmond, Washington, USA. E-mail: {\tt peres@microsoft.com}.}

\date{}
\begin{abstract}
A sequence of chains exhibits (total variation) cutoff  (resp., pre-cutoff) if for all $0<\epsilon< 1/2$, the ratio  $\mix^{(n)}(\epsilon)/\mix^{(n)}(1-\epsilon)$ tends to 1 as $n \to \infty  $ (resp., the $\limsup$ of this ratio is bounded uniformly in  $\epsilon$), where $\mix^{(n)}(\epsilon)$ is the $\epsilon$-total variation mixing time of the $n$th chain in the sequence.
We construct a sequence of bounded degree   graphs $G_n$, such that the lazy simple random walks (LSRW)   on $G_n$  satisfy the  ``product condition"  $\mathrm{gap}(G_n) \mix^{(n)}(\epsilon) \to \infty $  as $n \to \infty$, where $\mathrm{gap}(G_n)$ is the spectral gap of the LSRW on $G_n$ (a known necessary condition for pre-cutoff that is often sufficient for cutoff), yet this sequence does not exhibit pre-cutoff.

Recently, Chen and Saloff-Coste showed that total variation cutoff is equivalent for the sequences of continuous-time and lazy versions of some given sequence of chains. Surprisingly, we show that this is false when considering separation cutoff.  

We also construct a sequence of bounded degree graphs $G_n=(V_{n},E_{n})$ that does not exhibit cutoff, for which a certain bounded perturbation of the edge weights leads to cutoff and increases the order of the mixing time by an optimal factor of $\Theta (\log |V_n|)$. Similarly, we also show that ``lumping" states together may increase the order of the mixing time by an optimal factor of $\Theta (\log |V_n|)$. This gives a negative answer to a question asked by Aldous and Fill.
\end{abstract}

\maketitle


\section{Introduction}

Consider a  reversible irreducible lazy discrete-time Markov chain $X=(X_{t})_{t\ge 0}$, defined on a finite state space $\gO$ (we call a chain finite if $\Omega$ is finite).
Let $P$ and $\pi$ denote the transition matrix and the unique reversible probability measure associated to $X$, respectively (we  denote such a chain by  $(\Omega,P,\pi)$). In particular, the \emph{laziness} and \emph{reversibility} assumptions are (resp.) that $P(x,x) \ge 1/2 $ and $\pi(x)P(x,y)=\pi(y)P(y,x) $ for all $x,y \in \Omega$. To avoid periodicity and near-periodicity issues, one often considers
the lazy version of a discrete time Markov chain, $(X_t^{\mathrm{L}})_{t
= 0}^{\infty}$, obtained by replacing $P$ with $P_{\mathrm{L}}:=\frac{1}{2}(I+P)$. Periodicity issues can be avoided also by considering the continuous-time version of the chain, $(X_t^{\mathrm{c}})_{t
\ge 0}$.
This is a continuous-time Markov chain whose heat kernel is defined by $H_t(x,y):=\sum_{k=0}^{\infty}\frac{e^{-t}t^k}{k!}P^k(x,y)$.
It is a classic result of probability theory that for any initial condition
the distribution 
of both $X_t^{\mathrm{L}}$ and $X_t^{\mathrm{c}}$ converge to $\pi$ when $t$ tends to infinity. The object of the
theory of Mixing times of Markov chains is to
study the characteristic of this convergence (see \cite{cf:LPW}
for a self-contained
introduction to the subject).

%

\medskip

For any two distributions $\mu,\nu$ on $\Omega$,  their \emph{total variation
distance} is defined as
$\|\mu-\nu\|_\mathrm{TV} := \frac{1}{2}\sum_{x } |\mu(x)-\nu(x)|$.
The worst-case total variation distance at time $t$ is defined as $d(t)
:= \max_{x \in \Omega}   \|\Pr_{x}^t- \pi\|_\mathrm{TV}$, where we denote by $\Pr_{x}^t$ (resp.~$\Pr_{x}$) the distribution of $X_t$ (resp.~$(X_t)_{t
\ge 0}$), given that $X_0=x$.
The $\epsilon$ total variation \emph{mixing time} is defined as  $$t_{\mathrm{mix}}(\epsilon)
:= \inf \left\{t : d(t) \leq
\epsilon \right\}, $$

and for the continuous-time chain as $t_{\mathrm{c}}(\epsilon):=\inf \left\{t : \max_x \|H_t(x,\cdot)-\pi(\cdot) \|_{\TV} \leq
\epsilon \right\}$.

Similarly, the (worst-case) \emph{separation distance} from stationarity at time $t$ (resp.~for the continuous-time chain) is defined as 
$$\dsep(t):=1-\min_{x,y\in \gO} P^t(x,y)/\pi(y), \quad \dsep^{\mathrm{c}}(t):=1-\min_{x,y\in \gO} H_{t}(x,y)/\pi(y)$$
and the $\epsilon$ \emph{separation time} (the ``$\epsilon$ separation mixing time") is defined as
$$\sepeps:= \inf \left\{t : \dsep(t) \leq
\epsilon \right\}, \quad (\text{resp. }t_{\mathrm{sep,c}}(\epsilon):= \inf \left\{t : \dsep^{\mathrm{c}}(t)\leq
\epsilon \right). $$

When $\epsilon=1/4$ we omit it from the above
notation.

\medskip

In this work we consider lazy simple random walks (\textbf{LSRW}) on a sequence of finite (uniformly) bounded degree connected graphs, $G_n:=(V_n,E_n)$, whose sizes tend  to  infinity and also lazy random walks on a sequence of networks obtained from them via a bounded perturbation of their edge weights (we defer the formal definition to \S~\ref{s: bddpert}).

\medskip

Following \cite{cf:Ding}, where Ding and the second author showed that a bounded perturbation can increase the order of the total variation mixing time (we state their result in more details in \S~\ref{sec: Related}), we  study (by constructing relevant examples) the  possible effects of such bounded perturbations on the convergence to the stationary distribution (of the corresponding lazy random walks on the perturbed networks compared to the original LSRWs). In particular, our Theorem \ref{thm: sensitive2} asserts that such bounded perturbations can increase the order of the total variation mixing-times by an optimal (as explained in \eqref{eq: maxchange} below) factor  of $\Theta(\log |V_n|)$.

While the aforementioned result is merely an improvement and a simplification of  \cite[Theorem 1.1]{cf:Ding}, various aspects related to sensitivity of mixing times and the \emph{cutoff phenomenon} (of LSRW on a sequence of uniformly bounded degree graphs $G_n=(V_n,E_n)$) are considered in this work for the first time:
\begin{itemize}
\item[(1)]
We consider $o(1)$-perturbations (in which the weight of each edge may increase only by a $1+o(1)$ multiplicative factor) and show that they may increase the order of the mixing time by a factor whose order is arbitrarily close to $\log |V_n|$ (part (b) of Theorem \ref{thm: sensitive2}).

\medskip

\item[(2)]
We consider sensitivity of mixing-times under \emph{lumping} (Definition \ref{def:lumping}). We show that (even in the bounded degree unweighted setup) lumping together pairs of vertices may increase the order of the mixing time by an optimal factor  of $\Theta(\log |V_n|)$ (part (c) of Theorem \ref{thm: sensitive2}). This provides a negative answer to a question by Aldous and Fill \cite[Problem 4.45]{cf:Aldous} (Problem \ref{prob:AF} here).

\medskip

\item[(3)] We show that (in the above setup) the mixing time of the lazy non-backtracking random walk may be larger than that of the LSRW by a factor of  $\Theta(\log |V_n|)$ (Remark \ref{rem:NBRW}). A similar example (in which the ratio of the mixing times is   $o(\log |V_n|)$) was recently constructed by Hubert Lacoin et al.~during an AIM workshop on mixing times of Markov chains.

\medskip

\item[(4)] We show that the occurrence/non-occurrence of cutoff/pre-cutoff (see \eqref{def:cutoff}-\eqref{eq: precutoff}) is sensitive under $o(1)$ perturbations of the edge weights (Theorems \ref{thm sensitivity} and \ref{thm: sensitive2}). We also show that even in the above setup, the \emph{product condition} \eqref{eq: pcond} need not imply pre-cutoff (Theorem \ref{thm sensitivity}).

\medskip

\item[(5)] Perhaps our most surprising result (Theorem \ref{thm: ctslazy}) is that the occurrence of \emph{separation cutoff} \eqref{def:cutoff} for the sequence of discrete-time lazy chains, does not imply the same for the associated sequence of continuous-time chains (this can be interpreted as a "sensitivity" result w.r.t.~the choice of discrete/continuous time\footnote{A similar sensitivity holds w.r.t.~the choice of holding probability (see Remark \ref{rem: lazypqintro}).}). This is in contrast with the case of total variation cutoff  \cite{cf:Chen} due to Chen and Saloff-Coste.
\end{itemize}

 In \cite{Hermon2016} the first author constructed a sequence of pairs of 2-roughly isometric graphs $G_n=(V_n,E_n),G'_n=(V'_n,E'_n) $ of uniformly bounded degree with $|V_n| \to \infty $, whose $\ell_{\infty}$-mixing times differ by an optimal factor of  $\Theta(\log \log |V_n|)$. In this paper we study the convergence to the stationary distribution only w.r.t.~the total variation distance and the separation distance.

\subsection{The general moral of our results}
We now discuss the moral of our results. An important question is whether mixing times are robust. A related question is whether they can be characterized (perhaps only up to some universal constants and only under reversibility) using a geometric quantity which is robust. Different variants of this question were asked by various authors such as Pittet and Saloff-Coste \cite{pittet}, Kozma \cite[p.\ 4]{Kozma}, Diaconis and Saloff-Coste \cite[p.\ 720]{diaconis} and Aldous and Fill \cite[Open Problem 8.23]{cf:Aldous} (Kozma conjectured that  the  $\ell_\infty$-mixing time is robust under rough-isometries for LSRWs on bounded degree graphs, and the last two references ask for an extremal characterization of the $\ell_{\infty}$-mixing time in terms of the Dirichlet form).  

There are numerous works aiming at sharp geometric bounds on mixing-times such as the Fountoulakis-Reed bound \cite{cf:FR} (recently refined by Addario-Berry and Roberts \cite{Addario}) on the total variation mixing time and Morris and Peres' evolving sets bound \cite{cf:Evolving} on the $\ell_\infty$-mixing time, both expressed in terms of the expansion profile of the graph. The sharpest geometric bounds on the $\ell_\infty$-mixing time are given by the spectral profile, due to Goel et al.~\cite{cf:Spectral} and by the Log-Sobolev bound \cite[Corollary 3.11]{diaconis} due to Diaconis and Saloff-Coste. 
 Because these bounds involve geometric quantities, they are robust under small changes to the geometry, like bounded perturbations of the edge weights or (in the bounded degree unweighted setup) under rough-isometries. 

 Our results strengthen the  cautionary
note of Ding and Peres \cite{cf:Ding} (and also of \cite{Hermon2016}) on the possibility of developing a sharp geometric bound on mixing-times. Indeed, any sharp bound would have to distinguish in some cases between the LSRW
on a graph and the walk obtained by some $o(1)$-perturbation.

\medskip

Although receiving much attention, the investigation of the cutoff phenomenon has progressed mostly through the study of examples (or of certain classes of chains) rather than by developing general theory. Our results concerning sensitivity of the cutoff phenomenon (namely, Theorem \ref{thm: ctslazy} and parts of Theorems \ref{thm sensitivity} and \ref{thm: sensitive2}) demonstrate some difficulties in developing such general theory. 

Theorem \ref{thm: ctslazy} demonstrates that despite the fact that the separation and the total variation distances are intimately related to one another (e.g.~\cite[Eq.\ (1.5), (1.7) and (1.8)]{cf:Hermon}), the former may exhibit surprising behaviors which the latter cannot exhibit. For more on this point see \cite[Remark 1.4 and \S~2.4]{cf:Hermon}.   

\subsection{Cutoff and pre-cutoff}

Before stating our results concerning the cutoff phenomenon we must first give a few definitions.
Next, consider a sequence of chains, $((\Omega_n,P_n,\pi_n): n \in \N)$,
each with its corresponding
worst-distances from stationarity $d^{(n)}(t)$, $\dsepn(t)$, its mixing and separation times $t_{\mathrm{mix}}^{(n)}$, $\sepn $,
etc.. Loosely speaking, the total variation (resp.~separation) \emph{\textbf{cutoff phenomenon}}
 is said to occur when over a negligible period of time, known as the \emph{\textbf{cutoff
window}}, the worst-case total variation distance (resp.~separation distance) drops abruptly from a value
close to 1 to near $0$. In other words, one should run the $n$-th chain until
time $(1-o(1))\mixn $ (resp.~$(1-o(1))\sepn $) for it to even slightly mix in total variation (resp.~separation), whereas
running it any further after time $(1+o(1))\mixn $ (resp.~$(1+o(1))\sepn $) is essentially redundant.
Formally, we say that the sequence exhibits a \emph{\textbf{total variation cutoff}} (resp.~\emph{\textbf{separation cutoff}}) if the
following
sharp transition in its convergence to stationarity occurs:
\begin{equation}\label{def:cutoff}
\forall \gep\in (0,1/2], \quad  \lim_{n \to \infty}t_{\mathrm{mix}}^{(n)}(\epsilon)/t_{\mathrm{mix}}^{(n)}(1-\epsilon)=1 
\left( \text{ resp. }\lim_{n \to \infty}t_{\mathrm{sep}}^{(n)}(\epsilon)/t_{\mathrm{sep}}^{(n)}(1-\epsilon)=1\right).
\end{equation}

\medskip

We say that the sequence exhibits a \emph{\textbf{(total variation) pre-cutoff}}
 if
 \begin{equation}
 \label{eq: precutoff}
 \sup_{0<\epsilon<1/2}
\limsup_{n\to \infty} \mixneps/\mixn(1-\eps)< \infty.
\end{equation}

The notions of total variation and separation cutoff for the corresponding sequence of continuous-time chains are defined in an analogous manner ($\lim_{n \to \infty}t_{\mathrm{c}}^{(n)}(\epsilon)/t_{\mathrm{c}}^{(n)}(1-\epsilon)=1$ and $\lim_{n \to \infty}t_{\mathrm{sep,c}}^{(n)}(\epsilon)/t_{\mathrm{sep,c}}^{(n)}(1-\epsilon)=1 $, resp., for all $\epsilon \in (0,1) $). One can also consider the mixing time and separation-time for the sequence of associated lazy chains and define the two notions of cutoff for it. Recently, Chen and Saloff-Coste \cite{cf:Chen} showed that if $t_{\mathrm{c}}^{(n)} \to \infty$ then the sequence of the associated continuous-time chains exhibits total variation cutoff iff the same holds for the sequence of the associated lazy chains. A natural question (L.~Saloff-Coste, private communication) is whether the same is true for cutoff in separation. Surprisingly, this turns out to be false.
\begin{maintheorem}
\label{thm: ctslazy}
There exists a sequence of reversible chains so that the lazy chains exhibit separation cutoff but the associated continuous-time chains do not exhibit cutoff.
\end{maintheorem}
\begin{remark}
\label{rem:11}
The example we give for Theorem \ref{thm: ctslazy} is intimately related to Example 3 in \cite{cf:Hermon} and could be transformed into an example involving simple random walks on a sequence of bounded degree graphs using a similar construction as Example 5 in \cite{cf:Hermon}.
\end{remark}

\begin{remark}
\label{rem: lazypqintro}
The $\delta$-lazy version of a chain with transition matrix $P$ is obtained by replacing $P$ with $\delta I+(1-\delta)P$ (where $I$ is the identity matrix). Chen and Saloff-Coste \cite{cf:Chen} showed that given a sequence of chains, the corresponding sequence of $1/2$-lazy chains exhibits total variation cutoff iff the same holds for the corresponding sequence of $p$-lazy chains, for all $p \in (0,1)$. One can use the idea behind the construction from the proof of Theorem \ref{thm: ctslazy} in order to construct a family of reversible examples demonstrating that for all $p \neq q \in (0, 1) $ it is possible that the sequence of $q$-lazy versions of a certain sequence of chains exhibits separation cutoff, while the sequence of $p$-lazy versions does not exhibit separation cutoff or vice-versa\footnote{Our analysis shows that for all $p \neq q \in (0,1)$ it is possible to construct an example in which separation cutoff fails for exactly one of the corresponding lazy chains. We do not determine for which of them separation cutoff fails. Determining which of the two chains exhibits separation cutoff requires a more detailed comparison of certain two large deviation rate functions, where for our analysis it suffices to use the fact that the two rate functions are not identical.}. The necessary adaptations are described in \S~\ref{sec: lazypq}.
\end{remark}
\begin{question}
Is it the case that separation cutoff for the sequence of continuous-time chains implies the same for the sequence of lazy chains?
\end{question}
Recall that if $(\Omega,P,\pi)$ is a finite reversible irreducible lazy chain,
then $P$ is self-adjoint w.r.t.~the standard inner-product induced by $\pi$ (see Definition \ref{def: Linfty}) and hence has $|\Omega|$ real eigenvalues,
 $1=\lambda_1>\lambda_2 \ge \ldots \ge
\lambda_{|\Omega|} \ge 0$ (where $\lambda_2<1$ and $\lambda_{|\Omega|} \geq 0$ by irreducibility and laziness). The \emph{spectral-gap}
and   \emph{relaxation-time} are defined as $\mathrm{gap}:= 1-\lambda_2$ and $t_{\mathrm{rel}}:=(1-\lambda_2)^{-1}$, resp.. We denote the relaxation-time of LSRW  on $G=(V,E)$ by $\rel(G)$.

In 2004  \cite{peresamerican}, during an AIM workshop on the cutoff phenomenon, the second author introduced the so called \textbf{product condition}: \begin{equation}
\label{eq: pcond}
\reln =o(\mixn),
\end{equation}
a necessary condition for pre-cutoff (by \eqref{eq: t_relintro}; e.g.~\cite[Proposition 18.4]{cf:LPW}), and suggested that it is also a sufficient condition for cutoff for many ``nice" families of reversible chains. In general, the product condition does not always imply cutoff. Aldous and
Pak (private communication via P.~Diaconis) have constructed
relevant (reversible) examples (see \cite[Chapter 18]{cf:LPW}, Pak's example is described in \S~\ref{sec: Related}). This
left open the problem of identifying general classes of chains for which
the product condition is indeed sufficient for cutoff. This was verified e.g.~for lazy birth and death chains \cite{cf:bdcutoff} and recently for lazy weighted walks on trees \cite{cf:cutoff}.

\medskip

In Aldous' example the graph supporting the transitions is of bounded degree and contains only a single cycle, however the ratio of the maximal and minimal edge weights is exponentially large in the size of the state space. As noted in  \cite{cf:LS},  Aldous' example, which exhibits pre-cutoff, can be transformed into a sequence of LSRWs on bounded degree graphs (with pre-cutoff). Explicit constructions of such graphs (which were constructed as examples demonstrating that, in general, neither total variation cutoff nor separation cutoff  implies the other) can be found at \cite{cf:Hermon}.

  Until now, as in Pak's example, every known example in which the product condition does not imply pre-cutoff had unbounded degrees. These  examples all share the following behavior. The chain mixes (in some sense) ``at once" due to the occurrence of a certain rare event, which occurs before the chain has enough time to get even slightly mixed otherwise. It is plausible that (some concrete formulation of) the aforementioned behavior is necessary in order for pre-cutoff to fail when the product condition holds. Moreover, a-priori, it is not clear whether the mechanism that allows such behavior can be produced in the bounded degree (unweighted) setup. 

\medskip

A question, presented to us by E.~Lubetzky (private communication), which naturally arises in light of the above discussion,  is whether the product condition is also a sufficient condition for pre-cutoff for a sequence of LSRWs on bounded degree graphs  $\{ G_n \}_{n \in \N} $. One case in which this holds
(as a simple consequence of $\ell_2$-contraction; see Lemma \ref{lem: contraction}) is when $\{ G_n \}_{n \in \N} $ is a family of bounded degree expanders (that is, $\rel(G_n)=\Theta(1) $).

\begin{prob}[E.~Lubetzky (private communication)]
\label{prob: L}
Let $\{ G_n \}_{n \in \N} $ be a sequence of finite (uniformly)  bounded degree graphs satisfying the product condition. Is it always the case that the sequence of lazy simple random walks on $\{ G_n \}_{n \in \N} $  exhibits a pre-cutoff?
\end{prob}

Our  Theorem \ref{thm sensitivity} provides a negative answer to this question. Our construction may be viewed as a bounded degree (unweighted) version of the aforementioned Pak's example (see \S~\ref{sec: Related} and Remark \ref{rem: Geo} for further details concerning this point).

\subsection{Perturbations of edge weights and lumping}
\label{s: bddpert}
Let  $G=(V,E)$  be a finite connected simple graph. Given a (weighted) network $(V, E,(c_{e})_{e
\in E})$, where each edge $\{u,v\} \in E$ is endowed with a conductance $c_{u, v}=c_{v,u}>0$ (with the convention that $c_{u,v}=0$ if $\{u,v\} \notin E$),  a lazy random walk on $G=(V,E)$ repeatedly does the following: when the current state is $v\in V$,  the random walk will stay at $v$ with probability $1/2$ and move to vertex $u$ (such that $\{u,v\} \in E$) with probability $c_{u, v}/(2c_{v})$, where $c_v:=\sum_{w} c_{v, w}$.
The default choice for $c_{u, v}$ is 1 (in which case, we say that the random walk is {\em unweighted}), which corresponds to lazy simple random walk  on $G$ (in which at each step the walk with equal probability either stays put or moves to a new vertex, chosen from the uniform distribution over the neighbors of the current position). Its stationary distribution is given by $\pi(x):=c_x/c_V$, where $c_{V}:=\sum_{v \in V}c_v=2\sum_{e \in E}c_e$.
 We denote by $t_{\mathrm{mix},G,(c_{e})}(\eps)$ the total variation $\eps$-mixing time of the lazy random walk on the network induced on the graph $G$ by the edge weights $(c_e)_{e \in E } $. When the walk is unweighted we omit $(c_{e})_{e
\in E}$ from our notation. As always, when $\epsilon=1/4$ we omit it from the notation. In the unweighted setup we write $\mix (G)$ for the $(1/4)$-mixing time.

\begin{definition}
\label{def:lumping}
Given some network  $(V, E,(c_{e})_{e
\in E})$ with stationary distribution by $\pi$, consider a partition of $V$ into disjoint sets $A_1,\ldots,A_k$ with $\cup_{i=1}^k A_i=V$. The induced network, obtained by \emph{lumping} together the states belonging to $A_i$ for all $1 \le i \le k$ is a network on $[k]:=\{1,\ldots,k\}$ with transition probabilities $\hat p_{i,j}:=\Pr_{\pi}[X_1 \in A_j \mid X_0 \in A_i ]$. It can be obtained by collapsing each $A_i$ into a single state $i$ and setting the edge weight of the edge connecting $i$ and $j$ to be $\hat c_{i,j}:=\sum_{a_{i} \in A_i,a_{j} \in A_j }c_{a_i,a_j}$ for all $(i,j) \in [k] \times [k]$ (in particular, it is reversible).
\end{definition}
 Proposition 4.44 in \cite{cf:Aldous} asserts that several natural parameters of a reversible Markov chain, like the inverses of its spectral gap and cheeger constant, can only decrease as a result of lumping states together. This motivates the following question asked by Aldous and Fill  \cite[Open Problem 4.45]{cf:Aldous}.

\begin{prob}
\label{prob:AF}
Is it the case that lumping states together can increase the total variation mixing time of a reversible Markov chain by at most some constant factor $K$? Can one take $K=1$?
\end{prob}  
Our  Theorem \ref{thm: sensitive2} (part (c)) gives a negative answer to Problem \ref{prob:AF}.

 We say that $(w_e^{(n)})_{e \in E_n}$ is a sequence of \textbf{bounded perturbations} (respectively $o(1)$\textbf{-perturbations}) if $w_e^{(n)} \ge 1 $ for all $e \in E_n$ for all $n$ and $M((w_e^{(n)})):=\sup_n M_{n}< \infty  $ (resp.~$M_{n}=1+o(1)$), where $M_n:=\sup_{e \in E_n}w_e^{(n)}$.\footnote{Note that there is no loss of generality in the requirement that  $w_e^{(n)} \ge 1 $ for all $e \in E_n$, since multiplying all of the edge weights by the same constant has no effect on the distribution of the walk.}

\medskip

We say that $\{ G_n \}_{n \in \N} $ is {\bf robust} if performing a sequence of bounded perturbations, $(w_e^{(n)})_{e \in E_n}$ preserves the  mixing-times up to a constant factor of $K$ (independent of $n$, but which may depend on $M((w_e^{(n)}))$). We say that $\{ G_n \}_{n \in \N} $ is {\bf sensitive} if it is not robust. We say that   $\{ G_n \}_{n \in \N} $  is $o(1)$-\textbf{sensitive} if there exists a sequence of $o(1)$-perturbations which either increases or  decreases the order of the mixing-times.

\medskip

The girth of a graph $G$ is defined as the minimal length of a cycle in $G$. In general (even in the weighted case) $\mathrm{girth}(G) \le \mathrm{2diameter}(G) +1\le C \mix(G)$.

\begin{maintheorem}
\label{thm sensitivity}
For every $f(n)=o(\log n / \log \log n)$ such that $\lim_{n \to \infty}f(n)=\infty$,   there exists a sequence of bounded degree graphs $G_{n}=(V_{n},E_{_{n}})$ satisfying 
\begin{itemize}
\item[(i)]
$\mathrm{girth}(G_n)=\Theta(\mix(G_n)) $.

\medskip

\item[(ii)] The corresponding sequence of LSRWs satisfies the product condition but does not exhibit  pre-cutoff.

\medskip

\item[(iii)] There exists a sequence of $o(1)$-perturbations, $(c_e^{(n)})$, such that the corresponding sequence of lazy walks exhibits a cutoff and satisfies\begin{equation}
\label{eq: increasedtmix}
t_{\mathrm{mix},G_n,(c_{e}^{(n)})}/t_{\mathrm{mix}}(G_{n})=\Theta \left(f(|V_{n}|)\right).
\end{equation}
\end{itemize}
\end{maintheorem}

\medskip
The following remark explains the significance of the large girth condition.

\begin{remark}
\label{rem: girth}
In \cite{cf:cutoff} it was shown that a sequence of lazy random walks on weighted trees exhibits cutoff iff it satisfies the product condition. In \cite{cf:PS} it was proved that the mixing time of (possibly weighted) nearest neighbor lazy walks on trees is robust (see \cite[Theorem 1]{Addario} for a recent extension of this result). Combining the two results it follows that for lazy weighted walks on trees the property of exhibiting cutoff is robust. Theorem \ref{thm sensitivity} asserts that the tree assumption in these two results cannot be relaxed to the condition that $\mathrm{girth}(G_n)=\Theta(\mix(G_n)) $ (even in the unweighted setup, and even when considering only pre-cutoff, instead of cutoff).
\end{remark}

\begin{maintheorem}
\label{thm: sensitive2}
\begin{itemize}
\item[(a)]
There exists a sequence of bounded degree graphs $G_n=(V_n,E_n)$ satisfying $\rel(G_n)=\Theta (t_{\mathrm{mix}}(G_{n})) $ (thus lacking pre-cutoff) such that for every $\epsilon>0$, increasing the edge weight of some of the edges of $G_n$ to $1+\epsilon$  increases the mixing time by a factor of $c_{\epsilon} \log |V_n|$, for some constant $c_{\epsilon}>0$ depending only on $\epsilon$. Moreover, the sequence of walks on the perturbed networks exhibits a cutoff.

\medskip

\item[(b)] For every $f(n)=o(\log n)$ such that $\lim_{n \to \infty}f(n)=\infty$ there exists  a sequence of graphs $G_n=(V_n,E_n)$ satisfying $\rel(G_n)=\Theta (t_{\mathrm{mix}}(G_{n})) $  for which there exists a sequence of $o(1)$-perturbations which increases the mixing-times by a factor of $\Theta(f(|V_n|))$. Moreover, the sequence of lazy walks on the perturbed networks exhibits a cutoff.
\item[(c)] There exists a sequence of bounded degree graphs $G_n=(V_n,E_n)$ such that lumping together some pairs of vertices of $G_n$ increases the order of the total variation mixing time by a factor of $\Theta( \log |V_n|) $.
\end{itemize}
\end{maintheorem}

\begin{remark}
\label{rem: optimal}
We note that the $\log |V_{n}|$ factor in part (a) of Theorem \ref{thm: sensitive2} is optimal.
The following general relation holds for lazy reversible chains (e.g.~\cite{cf:LPW} Theorems 12.3-12.4).
\begin{equation}
\label{eq: t_relintro}
(t_{\mathrm{rel}}-1)|\log (2\epsilon)| \le t_{\mathrm{mix}}(\epsilon) \le  t_{\mathrm{rel}}|
\log \left( \epsilon \min_x \pi(x) \right)|.
\end{equation}

\medskip

Let $G=(V,E)$ be a graph of maximal degree $D$. Let  $(c_{e})_{e \in E}$ satisfy $K^{-1} \le c_e \le K $ for all $e \in E$. By the extremal characterization of the spectral gap via the Dirichlet form (e.g.~\cite[Remark 13.13]{cf:LPW}), the relaxation-time is robust under bounded perturbations (i.e.~a bounded perturbation can change it only by a constant factor). Consequently, by (\ref{eq: t_relintro}),
\begin{equation}
\label{eq: maxchange}
1/(C_{D,K}\log |V |) \le t_{\mathrm{mix},G,(c_{e})} /t_{\mathrm{mix}}(G)\le C_{D,K}\log |V |,
\end{equation}
for some constant $C_{D,K}$ depending only on $D$ and $K$. Similarly, since as mentioned earlier lumping does not increase $\rel$, also part (c) is optimal, up to a constant factor. 
\end{remark}
\begin{remark}
We note that in part (a) of Theorem \ref{thm: sensitive2} we have that $\mix(G_n)=\Theta( \log |V_n|)$, and so the mixing time of the perturbed network is $\Theta([\mix(G_n)]^2) $. This is also optimal by (\ref{eq: t_relintro}) and the bound $\mathrm{diameter}(G_n) \ge c \log |V_n|$ (for some $c>0$ depending only on the maximal degree).
\end{remark}
\begin{remark}
If the function $f$ in part (b) above is taken to tend to infinity sufficiently slowly, we can have that the $o(1)$-perturbation from part (b) increases the edge weight by a factor of $1+\delta_n$ for some $\delta_n=o(1)$ such that $\delta_n \sqrt{t_{\mathrm{mix}}(G_{n})} $ tends to infinity arbitrarily slowly. An interesting problem is to determine how small can $\delta_n$ be taken in terms of $\mix(G_n)$ (or to construct such an example with $\delta_n=o([\mix(G_n)]^{-\alpha})$ for some $\alpha > 1/2$).
\end{remark}

\begin{remark}
It is interesting to note that in all of our examples other than the one in the proof of Theorem \ref{thm: ctslazy}, up to a constant factor, the mixing time of the LSRW is independent of the starting position of the walk (in fact, it is $\Theta(\log |V|)$ in all of our examples).
\end{remark}

\subsection{Open problems}
Let $\{ G_n \}_{n \in \N} $ be a sequence of constant degree vertex-transitive graphs.
\begin{prob}[Re-iterated from \cite{cf:Ding}] Is $\{ G_n \}_{n \in \N} $ robust? 
\end{prob}

\begin{prob}  Assume that $\{ G_n \}_{n \in \N} $  satisfies the product condition. Is it true that the corresponding sequence of LSRWs must exhibit a pre-cutoff? If so, is the pre-cutoff robust? 
\end{prob}
Similarly, one can consider robustness w.r.t.~rough isometries.

\subsection{Notation}
For every $n \in \N$ we denote $[n]:=\{1,2,\ldots,n\}$. For any $a,b \in \R$ we write $a \vee b:=\max (a,b)$ and $a \wedge b:=\min
(a,b)$. Throughout, we use $C,C',C_0,C_1,\ldots$ and $c,c',c_0,c_1,\ldots$ to denote positive absolute constants that may be different from place to place. Given some parameter, say $\epsilon$, we write $C_{\epsilon}$ and $c_{\epsilon}$ for positive constants which depend only on $\epsilon$.
 Upper (resp.~lower) case letters will be used to denote sufficiently large (resp.~small) constants.

\section{Related Constructions}
\label{sec: Related}

As mentioned earlier, Ding and Peres have already constructed a sequence of sensitive bounded degree graphs \cite{cf:Ding}.
More precisely, for all $j \in \N $ they constructed a sequence of bounded degree graphs $G_n=(V_n,E_n)$ for which if some of the edge weights are doubled, then the order of the mixing times increases by a multiplicative factor of order $ \log |V_n|/ \log^{(j)}|V_n|$, where $ \log^{(j)}$ is the iterated logarithm of order $j$ (see \cite[Remark 2.3]{cf:Ding})\footnote{The construction is obtained by iterating the construction of the case $j=2$ and is hence somewhat more involved.}.

Our constructions use a key observation from \cite{cf:Ding}. Namely, we use the fact that the harmonic measure is sensitive under perturbations and that (as explained below) this can lead to sensitivity of mixing-times. This idea was originally used by Benjamini \cite{cf:Benjamini} to study instability of the Liouville property. We note that our construction in the proof of Theorem 2 was greatly influenced by Ding and Peres' construction and is intimately related also to the construction from \cite{Hermon2016} of a sequence of graphs of uniformly bounded degree whose $\ell_{\infty}$ mixing time is sensitive under bounded perturbations.

\medskip

The first construction of a sequence of finite irreducible lazy reversible chains satisfying the product condition,
which
does not exhibit pre-cutoff is due to Pak (private communication through Persi Diaconis, see \cite[Example 18.7]{cf:LPW}). Pak's construction gives a general scheme of constructing such sequences
of Markov chains. Start with a sequence of lazy reversible  chains
$(\Omega_n,P_n,\pi_n)$ which exhibits cutoff ($\pi_n P_n=\pi_n $). Let $\Pi_{n}$ be the
transition matrix whose rows all equal $\pin$. Denote $L_n:=\sqrt{\reln \mixn} $. Then the sequence $(\Omega_n,(1-\frac{1}{L_{n}})P_n+\frac{1}{L_{n}}\Pi_{n},\pi_n) $ satisfies
the product condition but does not exhibit a pre-cutoff.

Loosely speaking, the chain mixes ``at once", at a random time having a Geometric distribution with mean $L_n$, due to the occurrence of one rare event (moving according to $\Pi_{n}$) which (with high probability) occurs before the chain has enough
time to get even slightly mixed otherwise. At first sight, it is surprising that it is possible to construct an example of bounded degree graphs so that the corresponding sequence of LSRWs imitates this behavior. In our examples, mixing  occurs quickly once the chain reaches its ``center of mass", which is an expander. This allows us to reduce the analysis of the mixing time and the occurrence/non-occurrence of cutoff to the easier problem of analyzing the distribution of the
hitting time of the center of mass (namely, the mixing time is roughly equal to its mean, and total variation cutoff is equivalent to it being concentrated around its mean).

\medskip

In the construction from the proof of Theorem \ref{thm sensitivity}, the distribution of the hitting time of the ``center of mass" (starting from the worst starting state) would be  roughly a geometric distribution. Loosely speaking, starting from the worst starting state, until the time  the center of mass is reached, the chain looks like a LSRW on a regular tree whose edges were stretched (i.e.~replaced by a long path whose length tends to infinity) with some ``shortcuts" to the center of mass (see Figures \ref{fig:4}-\ref{fig:7}). The amount and positions of these shortcuts can be chosen so that the center of mass is reached (with high probability) through one of these shortcuts  at a random time having roughly a geometric distribution. Moreover, this can be done so that, under a certain perturbation of the edge weights, the harmonic measure is changed in a manner which makes these shortcuts ``invisible" to the walk. Hence after the perturbation, the walk (starting from the root) is ``trapped" in the tree of stretched edges for a much longer period of time, which results in an increased mixing time.

\medskip

The idea behind the construction of the graphs $G_n$ from Theorem \ref{thm sensitivity} is simple. Start with an arbitrary sequence of constant degree expanders $H_n=(V(H_n),E(H_n))$ (with $|V(H_n)|\to \infty$) whose girth is $\ell_n:=\Theta(\log |V(H_n)|)$ (see e.g.~\cite{cf:Xpq} for the existence of such graphs). We pick some vertex $o \in V(H_n)$ and choose a certain collection of vertices $D$, all within distance $< \ell_n/2 $ from $o$. Finally, for all $d \in D$ we replace each edge along the shortest path between $o$ and $d$ by a path of length $s_n$, where $s_n \to \infty$ and $s_n^2=o(\ell_n/\log \ell_n)$. With some care, the set $D$ can be chosen (in a canonical manner) so that:
\begin{itemize}
\item The asymptotic profile of convergence in total variation of the walk can be understood in terms of the distribution (under $\Pr_o$) of the escape time from the collection of vertices which are incident to the stretched edges.

\item This escape time distribution is ``close" to the Geometric distribution with mean $\Theta(\ell_n)$. However, under a certain (canonical) $o(1)$-perturbation, this escape time distribution becomes concentrated around some time $t_n=\Theta(s_n^2\ell_n)$.
\end{itemize}

In \cite{cf:LS} Lubetzky and Sly gave an explicit construction of 3-regular
expanders, $G_n$, such that the sequence of lazy simple random walks on $G_n$
exhibits total variation cutoff. Our construction in the proof of Theorem \ref{thm sensitivity} resembles their construction (namely, in \cite{cf:LS} they also ``stretch" some of the edges of an expander).

\section{Proof of Theorem \ref{thm: ctslazy}}
\subsection{Preliminaries}
Throughout this section we denote the distribution of the associated lazy (resp.~continuous-time) chain started from $x$ by $\Pr_x$ (resp.~$\mathrm{H}_x $). We denote the transition matrix of the non-lazy (resp.~lazy) version of the chain by $P$ (resp.~$P_{\mathrm{L}}$). We denote the separation distance at time $t$ of the continuous-time and lazy chains by $d_{\mathrm{sep,c}}(t)$ and $d_{\mathrm{sep,L}}(t)$, respectively.

Before presenting the construction for Theorem \ref{thm: ctslazy} we first provide some technical machinery, borrowed from \cite{cf:Hermon}, which shall assist us in analyzing the asymptotic profile of convergence in separation. To characterize the separation time, we introduce a notion of ``double-hitting time".

\begin{definition}
\label{def: convolution}
Given $x,y$ and $z$ in $\gO$. We let $T_{z}^{x,y}$ (resp.~$\tau_z^{x,y} $) denote a random variable obtained by taking the sum 
of two independent realizations of $T_z:=\inf \{t:X_t=z \}$, once under $\mathrm{P}_x$  and once
under $\mathrm{P}_y$ (resp.~$\mathrm{H}_x$ and $\mathrm{H}_y$). More explicitly, we consider realizations of the lazy and continuous-time chains started from $x$ and $y$, denoted resp., by $(X_t^x)$ and $(X_t^{x,\mathrm{c}})$ (resp.~$(X_t^y)$ and $(X_t^{y,\mathrm{c}})$) defined on the same probability space, so that $(X_t^x)$ and $(X_t^y)$ and also   $(X_t^{x,\mathrm{c}})$ and $(X_t^{y,\mathrm{c}})$ are independent. Denote by $T_u^x:=\inf \{t:X_t^x=u \} $ (resp.~$\tau_u^x:=\inf \{t:X_t^{x,\mathrm{c}}=u \} $) the hitting time of state $u$ by $(X_t^x)$ (resp. $(X_t^{x,\mathrm{c}}) $). Define $T_u^y$ and $\tau_u^y $ in an analogous manner. We set \[T_{z}^{x,y}:=T_z^x+T_z^y, \quad \tau_z^{x,y}:=\tau_z^x+\tau_z^y .\] We define \[T_{z,y}^{x}:=T_y^x \ind_{\{T_y^x \le T_z^x \}}+(T_{z}^x+T_z^y)\ind_{\{T_y^x> T_z^x \}} \]  \[\tau_{z,y}^{x}:=\tau_y^x \ind_{\{\tau_y^x \le \tau_z^x \}}+(\tau_{z}^x+\tau_z^y)\ind_{\{\tau_y^x> \tau_z^x \}}.\] 

Finally, we denote the density function of $\tau_z^{x,y}$ by $f_{z}^{x,y}$ and that of the sub-distributions of $\tau_y^x \ind_{\{\tau_y^x \le \tau_z^x \}}$ and $(\tau_{z}^x+\tau_z^y)\ind_{\{\tau_y^x> \tau_z^x \}}$ by $g_{z,y}^{x}$ and $h_{z,y}^x$, respectively.
\end{definition}

The following lemma is a slight variation of Lemma 3.5 from \cite{cf:Hermon}. We present its proof in \S~\ref{s:lem32}  for the sake of completeness.
\begin{lemma}
\label{lem: pathsdecomposition}
Let  $(\Omega,P,\pi)$ be a finite   reversible Markov
chain. Consider $x,y, z
\in \Omega$.
\begin{itemize}
\item[(i)] For all $t \ge 0 $
we have that
\begin{equation}
\label{eq: septhroughz1}
\begin{split}
& P_{\mathrm{L}}^t(x,y)/\pi(y)
 = \sum_{k:k \le t }\mathbb{P}[T_{z,y}^{x}=k,T_y^x \le T_z^x ]P_{L}^{t-k}(y,y)/ \pi(y) \\ & +  \sum_{k:k \le t }\mathbb{P}[T_{z,y}^{x}=k,T_y^x > T_z^x ]P_{\mathrm{L}}^{t-k}(z,z)/ \pi(z)
\ge   \mathbb{P}[T_{z,y}^{x} \le t].
\\  H_t(x,y)/&\pi(y)
 =  \int_{0}^{t}(\frac{g_{z,y}^{x}(s)H_{t-s}(y,y)}{\pi(y)}+\frac{h_{z,y}^{x}(s)H_{t-s}(z,z)}{\pi(z)})ds\ge   \mathbb{P}[\tau_{z,y}^{x} \le t].
\end{split} 
\end{equation}
In particular,
\begin{equation}
\label{eq: septhroughz3}
P_{\mathrm{L}}^t(x,y)/\pi(y) \ge \Pr_x[T_y \le t], \quad H_t(x,y)/\pi(y) \ge \mathrm{H}_x[T_y<t].
\end{equation}
\item[(ii)]
If
$\mathrm{P}_{x}[T_z \le T_y]=1$ (i.e.~if every path from $x$ to $y$ goes through $z$) then for all $t \ge 0$
\begin{equation}
\begin{split}
\label{eq: septhroughz2}
& P_{\mathrm{L}}^t(x,y)/\pi(y)
=\sum_{k:k \le t }\mathbb{P}[T_{z}^{x,y}=k]P_{\mathrm{L}}^{t-k}(z,z)/ \pi(z),\\
& H_t(x,y)/\pi(y)
 = \int_{0}^{t}f_z^{x,y}(s)H_{t-s}(z,z)ds / \pi(z) \\ & \le \mathbb{P}[\tau_{z}^{x,y} \le t]+ t_{\mathrm{rel}} \max_{s} f_z^{x,y}(s) (1-\pi(z))/\pi(z).
\end{split}
\end{equation}
In particular, if $f_y^x$ is the density of the hitting time of $y$ from $x$, then
\begin{equation}
\label{eq: septhroughz4}
H_t(x,y)/\pi(y) \le \h_x[T_y \le t]+  t_{\mathrm{rel}} \max_{s} f_y^{x}(s) (1-\pi(y))/\pi(y).  
\end{equation}
\end{itemize}
\end{lemma}
The following example, Example \ref{ex:33} (a birth and death chain of size $2n+1$ with a fixed bias towards its middle point), will serve as a gadget in the construction for Theorem \ref{thm: ctslazy}. In Example \ref{ex:33} the state $z$ serves as the center of mass and has a $\Theta(1)$ stationary probability. The chain from Example \ref{ex:33} exhibits cutoff in separation around the expectation of the ``double hitting time" of $z$ from $a$ and $b$.

\begin{example}
\label{ex:33}
Let $\gO:=A \cup B \cup \{z\} $, where $A:=\{a_1,a_2,\ldots,a_n\} $, $B:=\{b_1,b_2,\ldots,b_n\} $. Denote $a_0:=z=:b_0$, $a:=a_n$ and $b:=b_n$. Consider the transition matrix $P(a,a_{n-1})=1=P(b,b_{n-1}) $, $P(a_{i},a_{i-1})=P(b_{i},b_{i-1})=2/3=2P(a_{i},a_{i+1})=2P(b_{i},b_{i+1}) $, for $1 \le i <n$ and $P(z,a_1)=1/2=P(z,b_1)$. Let $\delta>0$. Define
\begin{equation}
\label{eq: tdelta}
t_{\delta}=t_{\del}^{(n)}:=\min \{t: \mathbb{P}[T_z^{a,b} \le t] \ge 2^{-\delta n} \}, \quad \tau_{\delta}=\tau_{\del}^{(n)}:=\min \{t: \mathbb{P}[\tau_z^{a,b} \le t] \ge 2^{-\delta n} \}.
\end{equation}
Using large deviations theory (cf.~\cite[Example 3]{cf:Hermon}), it is not hard to show that $\lim_{n \to \infty} \tau_{\delta}/n$ and $\lim_{n \to \infty} t_{\delta}/n$ both exist. Moreover,  $2 \tau_{\del}<t_{\del} $ (for all sufficiently large $n$) and uniformly in $\delta \in (0,1/2) $, we have that   $\frac{12n-t_{\del}}{\sqrt{\del}n} =\Theta(1)=\frac{6n-\tau_{\del}}{\sqrt{\del}n} $  and 
\begin{equation}
\label{eq: tdel}
\frac{t_{\del}-2 \tau_{\del}}{\sqrt{\del}n}= \Theta(1).
\end{equation} 
\end{example}
Equation \eqref{eq: tdel} shall play a crucial role in the proof of Theorem \ref{thm: ctslazy}. The proof of \eqref{eq: tdel} is deferred to \S~\ref{s:ap64}.
\begin{figure}[!h]
\begin{center}
\includegraphics[height=3cm,width=16cm]{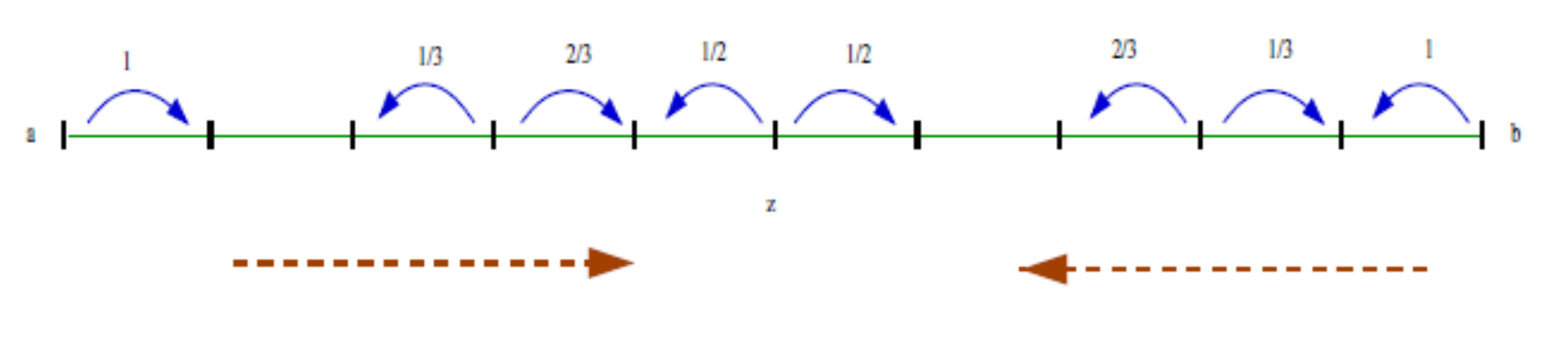}
\end{center}
\caption{\label{fig:basicchain} 
A schematic representation of the transition probabilities for the simple chain from Example \ref{ex:33}.  The endpoints $a$ and $b$ are both of distance $n$ from $z$.    In continuous-time the separation mixing time is twice as large as the total variation mixing time ($3n$ and $6n$, respectively).
The transition probabilities (apart from at the special states $a$, $b$ and $z$) are $2/3$ in the $z$ direction and $1/3$ in the opposite one.}
\end{figure}
\begin{remark}
\label{rem:34}
The reader is encouraged to consider a variant of  Example \ref{ex:33} in which the transition probabilities in are taken to be $1-e^{-n}$ towards $z$ and $e^{-n}$ away from $z$ (and the definitions of $t_{\delta}$ and $\tau_{\delta}$ remain unchanged, as in \eqref{eq: tdelta}).\footnote{The term $e^{-n}$ can be replaced by any other term which is $o(n)$. Effectively, it is as if the transition towards $z$ equals 1, but if we would have defined it to equal 1 then the chain would be reducible.} In this variant $\frac{4n-t_{\del}}{\sqrt{\del}n} =\Theta(1)=\frac{2n-\tau_{\del}}{\sqrt{\del}n} $ and (as in \eqref{eq: tdel})  $\frac{t_{\del}-2\tau_{\del}}{\sqrt{\del}n} =\Theta(1)$ (uniformly in $\delta \in (0,1/2)$), and the proofs of these equations   are essentially reduced to  a comparison of the large deviation behaviors of the Poisson and of the Binomial distributions. Namely, if $s_1=s_1(n,\delta ):=\inf \{t: \mathbb{P}[\mathrm{Pois}(t) \ge 2n] \ge 2^{- \delta n} \} $ and $s_2=s_2(n,\delta ):=\inf \{t: \mathbb{P}[\mathrm{Bin}(t,1/2) \ge 2n] \ge 2^{- \delta n} \}$, then (after changing the transition probabilities as above) we have that $\tau_{\delta}=(1\pm o(1))s_1 $ and  $t_{\delta}=(1\pm o(1))s_2$. For simplicity, instead of analyzing $s_1$ and $s_2$ we now analyze related quantities.  Let $X \sim \mathrm{Pois}(  n) $ and  $Y \sim \mathrm{Bin}(2  n,1/2)$.  Pick $t_1=t_1(n,\delta)$ and $t_2=t_{2}(n,\delta)$ such that $\Pr[X \le t_1 ] = \Theta(2^{-\delta n})= \Pr[Y \le t_2 ]$. Then  uniformly in $\delta \in (0,1/2) $, we have that   $\frac{n-t_1}{\sqrt{\del}n} =\Theta(1)=\frac{n-t_2}{\sqrt{\del}n} $  and $\frac{t_2-t_1}{\sqrt{\delta}n}=\Theta(1) $. While this can be derived directly via a comparison of the large deviation rate functions of the corresponding distributions,  the last equality can also be deduced from the fact that by Poisson thinning if for some $Z \sim \mathrm{Pois}(2n)$ we have that (given $Z$) $X \sim \mathrm{Bin}(Z,1/2)$, then $X \sim \mathrm{Pois}(  n) $ (for further details cf.~the proof of \eqref{eq: tdel}).  
\end{remark}

We now briefly explain how Example \ref{ex:33} will be used as a building block in the proof of Theorem \ref{thm: ctslazy}. 
As noted in \cite[Example 3]{cf:Hermon}, by attaching two birth and death chains (``branches") of length $\Theta(n)$ to $z$ both having the same end-points $z,z'$ with a bias towards $z'$, we can tune the stationary measure of $z$ to become exponentially small in $n$.\footnote{We note that in \cite[Example 3]{cf:Hermon} the roles of $z$ and $z'$ are interchanged compared to their role here.} We pick one of the branches to have a larger average speed than the other. In the example from the proof of Theorem \ref{thm: ctslazy} the slower branch is $C$ and the faster is $D$. For technical reasons, in our construction the end-points of the branches $C$ and $D$ are $z$ and $\bar z$, where $\bar z$ is connected to $z'$ through a biased birth and death chain ($E$) with a fixed bias towards $z'$ (see Figure \ref{fig:ex3}). Using ideas from \cite{cf:Hermon} we can tune simultaneously\footnote{The fact that we can tune them simultaneously is subtle and crucial.} both the stationary measure of $z$ (we will take it to be $\Theta(2^{-\delta n})$, for some $\delta \in (0,1/8) $) and the time it takes to reach $z'$ from $z$ along each branch. For each branch, if we condition on reaching $z'$ through that branch, the (conditional) distribution of the hitting time of $z'$ becomes concentrated. 

If the time it takes the chain to reach $z'$ from $z$ along the slow branch is sufficiently small, then the lazy and continuous-time chains would exhibit separation cutoff  around the time $t_n$ (resp. $\tau_n$) in which $\mathbb{P}[T_z^{a,b} \le t_{n}] = \Theta (\pi (z)) $ (resp.~$\mathbb{P}[\tau_z^{a,b} \le \tau_n] = \Theta (\pi (z))$). Indeed this is proven in \cite{cf:Hermon} Example 3, by exploiting the equality in \eqref{eq: septhroughz2} above and analyzing the distribution of $T_z^{a,b}$ in the large deviation regime. Note that because we will have that  $\pi(z)= \Theta(2^{-\delta n})$, the times $t_n$ and $\tau_n$ coincide (up to smaller order terms) with $t_{\del}$ and $\tau_{\del}$ (respectively) from \eqref{eq: tdelta}.

The significance of \eqref{eq: tdel} is that after modifying Example \ref{ex:33} as described above  so that   $\pi(z)= \Theta(2^{-\delta n})$, we get by \eqref{eq: tdel}  that  $\liminf_{n \to \infty} t_n/\tau_n=\liminf_{n \to \infty} t_{\delta}/\tau_{\delta} > 2 $ (where $t_n$ and $\tau_n$ are defined in the previous paragraph and  $t_{\del}$ and $\tau_{\del}$ in \eqref{eq: tdelta}). This  allows us to ruin separation cutoff for the continuous-time chain while maintaining it for the lazy chain, by picking the average speed along the two branches wisely, so that the hitting time of $z'$ (for the continuous-time chain), starting from $a$, is with a constant probability between $\tau_n+\eps n $ and $t_n/2$ (for some $\eps \in (0,\frac{1}{n}(\frac{t_n}{2}-\tau_n) ) $) and with the complement probability (up to negligible terms) is smaller than $\tau_n $. In other words, we exploit  \eqref{eq: tdel} 
to construct a Markov chain such that the following hold:
\begin{itemize}
\item[(1)]
 In terms of the separation distance it suffices to consider the case that the initial state is $a$. 
\item[(2)] The state $y$ which up to $o(1)$ terms minimizes $H_t(a,y)/\pi(y)$ and $P_{\mathrm{L}}^{2t}(a,y)/\pi(y)$   is different for $t$ lying in some (sufficiently large) time interval. 
\item[(3)]
For the continuous-time chain the worst state (i.e.\ the minimizer from (2))  is $z'$ for every $t$ in the aforementioned time interval.   Moreover, we will show that $H_t(a,z')=\h_{a}[T_{z'}>t]\pm o(1) $ for all $t$. For the sake of simplicity let us assume that the aforementioned time interval is $(0,\infty)$. In this case, the existence of two parallel branches of different average speeds through which $z'$ can be accessed prevents cutoff for the sequence of continuous-time chains. 
\item[(4)] For the discrete-time lazy chain we have that  $\min_{y} P_{\mathrm{L}}^{t}(a,y)/\pi(y)=P_{\mathrm{L}}^{t}(a,b)/\pi(b) + o(1)$ for all $t$. Moreover, we will show that $P_{\mathrm{L}}^{t}(a,b)/\pi(b)  $ exhibits an abrupt transition from $o(1)$ to at least $1-o(1)$ around time $t_{\delta}$ and hence the sequence of lazy chains exhibits cutoff around time $t_{\delta}$.
\item[(5)] The mechanism which allows us to construct an example with such behavior is that while we always have that  $ \Pr_a(T_{z'} >2 t)= \h_a(T_{z'} > t) \pm o(1) $, it is not the case that $\frac{1}{2}  t_{\delta}=\tau_{\delta}(1 \pm o(1))$. 
\end{itemize}

\subsection{Proof of Theorem \ref{thm: ctslazy}} Below we intentionally omit all ceiling and floor signs and suppress the dependence on $n$ of some quantities, for the sake of notational convenience. Fix some $0<\epsilon < \del<1/8$ and $2 \le s \in \N$ so that
\begin{equation}
\label{eq: deldel's}
3(1+ \frac{\delta}{2}s^{2})n +4\eps n \le \tau_{\del}+2\eps n <  3(1+\frac{\delta}{2}(1+ s^2))n<   \frac{1}{2}  t_{\delta} - 2\eps n,
\end{equation}
for every sufficiently large $n$, where $t_{\delta}$ and $\tau_{\delta}$ are as in \eqref{eq: tdelta}. Note that by \eqref{eq: tdel} we can find such $\eps=\eps(\del) $ and $s=s(\del)$, provided that $\del$ is sufficiently small. We also note that the leftmost inequality in \eqref{eq: deldel's} is not as important as the other two.\footnote{Even if  $3(1+ \frac{\delta}{2}s^{2})n  \ge \tau_{\del} $ only minor changes to the proof below are needed. Namely, instead of arguing that  $d_{\mathrm{sep},c}^{(n)}(\tau_{\del} + \eps n/2) $ is bounded away from 1,  one would have had to argue that $d_{\mathrm{sep},c}^{(n)}(3(1+ \frac{\delta}{2}s^{2})n + \eps n/2)  $  is bounded away from 1.}

\medskip

Take $\gO:=A \cup B \cup C \cup D \cup E \cup \{z,\bar z,z' \}$, where  $A:=\{a_1,a_2,\ldots,a_n=a\} $, $B:=\{b_1,b_2,\ldots,b_n=b\} $, $D:=\{d_1,\ldots,d_{\delta n /2} \}$, $E:=\{e_1,\ldots,e_{\delta n /2} \}$ and $C:=\{c_{i,j}: 1 \le i \le  \del n/2, 1 \le j \le s  \} $ (see Figure \ref{fig:ex3}). Before specifying the transition probabilities we specify some properties that we want the construction to satisfy. The restriction of the chain to $A \cup B \cup \{z\}$ is precisely the chain from Example \ref{ex:33}, where also here the states 
 $a$ and $b$ are the end-points. The state $z'$ serves as the center of mass of the chain (i.e.\ $\pi(z')=\Theta(1)$).
Below we essentially show that for both the continuous-time and the discrete-time chains it suffices to consider the case that the initial state is $a$ and that started from $a$ the only two other relevant states for mixing in separation are $b$ and $z$. More precisely, below (combining \eqref{eq: potentialmin}-\eqref{eq: ab} with the analysis of the four cases in the proof of \eqref{eq: potentialmin}) we show that (uniformly in $t$)
\begin{equation}
\label{e:dsl}
d_{\mathrm{sep},\mathrm{L}}^{(n)}(t)=\max \{ \Pr_a(T_{z'} > t), 1-\frac{P_{\mathrm{L}}^t(a,b)}{\pi (b)} \} \pm o(1), \end{equation}
\begin{equation}
\label{e:dsc}
d_{\mathrm{sep,c}}^{(n)}(t)=\max \{ \h_a(T_{z'} > t), 1-\frac{H_t(a,b)}{\pi (b)} \} \pm o(1), \end{equation} and that $\frac{P_{\mathrm{L}}^t(a,b)}{\pi (b)}$ (resp.\ $\frac{H_t(a,b)}{\pi (b)}$) exhibits an abrupt transition from $o(1)$ to at least $1-o(1)$  around time $t_{\delta}$ (resp.\ $\tau_{\delta} $). Conversely, due to the existence of  two parallel branches $C$ and $D$ through which $z'$ can be accessed we have that neither $ \Pr_a(T_{z'} > t)$ nor $ \h_a(T_{z'} > t)$ exhibit an abrupt transition as functions of $t$. Namely, they decrease from $1-o(1)$ to $o(1)$ via two drops occurring for the continuous-time chain around times $3(1+\delta)n $ and $3(1+\frac{\delta}{2}(1+ s^2))n$ (and for the lazy chain around times  $6(1+\delta)n $ and $6(1+\frac{\delta}{2}(1+ s^2))n$).    


  The sets $C$ and $D$ serve as two parallel branches with different average speeds (as in the discussion following Remark \ref{rem:34}) connecting $z$ to $\bar z$, which in turn is connected to the center of mass, $z'$, via the segment $E$ (see Figure \ref{fig:ex3}). In order to ensure that the average speed along $C$ is slower than along $D$ we subdivide its edges into paths of length $s$ (see Figure \ref{fig:ex3}).\footnote{We could have simply increased the holding probability along the vertices of the branch $C$. However, in order to make the details behind Remark \ref{rem:11} more transparent we choose  to subdivide its edges instead.} The term  $3(1+\frac{\delta}{2}(1+ s^2))n $ in \eqref{eq: deldel's} corresponds to the time around which the hitting time of $z'$ (for the continuous-time chain, started from either $a$ or $b$) is concentrated, given that $z'$ is reached through the slow branch $C$ (this is explained in more details below).  The aforementioned roles of the  terms in \eqref{eq: deldel's} (described over the last two paragraphs, other than that of the term   $3(1+ \frac{\delta}{2}s^{2})n$, whose significance can be seen from \eqref{yinC1} below) together with \eqref{e:dsl}-\eqref{e:dsc}   motivate \eqref{eq: deldel's}. 

 Indeed, by \eqref{eq: deldel's}  (and the aforementioned role of the term $3(1+\frac{\delta}{2}(1+ s^2))n $) we have  $\Pr_a(T_{z'} > t_{\delta} - 2\eps n)=o(1) $, and so by \eqref{e:dsl} we have that  $d_{\mathrm{sep},\mathrm{L}}^{(n)}(t)= \max\{0, 1-\frac{P_{\mathrm{L}}^t(a,b)}{\pi (b)} \} \pm o(1)$, which as mentioned below \eqref{e:dsc} exhibits an abrupt change from $1-o(1)$ to $o(1)$ around time $t_{\del}$. Thus the lazy chains exhibit separation cutoff around time $t_{\del}$. Similarly, by \eqref{eq: deldel's}  \[\h_a(T_{z'} > \tau_{\del}+2 \eps n ) \ge \h_a(z' \text{ is reached through the branch }C)-o(1),\] 
 \[\h_a(T_{z'} > \tau_{\del}+\eps n ) \le \h_a(z' \text{ is reached through the branch }C)+o(1),\]
which (as will be made clear) by construction is bounded away from 0 and 1 (where we say that $z'$ is reached through $C$ if the last state visited before $T_{\bar z}$ is in $C$, see Figure \ref{fig:ex3}). Hence by \eqref{e:dsc} and the  comment following it, for $t \in [\tau_{\del}+ \eps n ,\tau_{\del}+2 \eps n ]$ we have that $d_{\mathrm{sep,c}}^{(n)}(t)= \h_a(T_{z'} > t)\pm o(1)$ is bounded away from 0 and 1. Thus the continuous-time chains do not exhibit separation cutoff.   

The lengths of the branches $C$ and $D$ (and also of the interval $E$) are taken so that  $ \pi(z)= \Theta(2^{-\delta n})$. Note that this means that $t_{\delta} $ and $\tau_{\delta}$ agree (up to negligible terms) with  $t_n$ and $\tau_n$, respectively, from the discussion following Remark \ref{rem:34}. In light of the discussion following Remark \ref{rem:34}, this explains why (as mentioned above)  $\frac{P_{\mathrm{L}}^t(a,b)}{\pi (b)}$ (resp.~$\frac{H_t(a,b)}{\pi (b)}$) exhibits an abrupt transition around $t_{\delta}$ (resp.~$\tau_{\delta} $).\footnote{While the abrupt transition of $\frac{P_{\mathrm{L}}^t(a,b)}{\pi (b)}$ around time $t_n$ essentially follows from the analysis in \cite[Examples 3 and 5]{cf:Hermon}, we will  prove it and the abrupt transition of $\frac{H_t(a,b)}{\pi (b)}$ around time $\tau_n$ for the sake of completeness.   }  We now specify the edge weights\footnote{We write the edge weights (instead of the transition probabilities, which are given in Figure \ref{fig:ex3})   in order to demonstrate that the chain is indeed reversible and to facilitate the calculation of $\pi(z)$.} and introduce some additional notation (for a schematic representation of the transition probabilities see Figure \ref{fig:ex3}). 
\begin{itemize}
\item
 $a:=a_n$,  $b:=b_n$ (the states 
 $a$ and $b$ have symmetric roles in our construction). 
\item
For all $i$, $c_{i+1,0}:=c_{i,s}$. 
\item
  $a_0=b_{0}:=z=d_{1+\delta n/2}$ and identify $c_{\delta n/2,s }$ with $z$.
\item
 $c_{1,0}:=\bar z=:e_{1+\delta n /2}=d_0  $ and $e_0:=z'$.
\end{itemize}

Consider the following (symmetric) edge weights:
\begin{itemize}
\item
$w(b_i,b_{i-1})=2^{-i}=w(a_i,a_{i-1}) $ for all $1 \le i \le n $.
\item
$w(c_{i,j},c_{i,j-1})=2^{\del n/2-i}$ for all $1 \le i \le \delta n/2$, $1 \le j \le s$.
\item
$w(d_i,d_{i+1})=2^{\del n/2-i}$ for all $0 \le i \le \delta n /2 $.
\item
$w(e_i,e_{i+1})=2^{\delta n -i}$ for all $0 \le i \le \delta n/2 $. 
\end{itemize}

\begin{figure}[!h]
\begin{center}
\includegraphics[height=8cm,width=12cm]{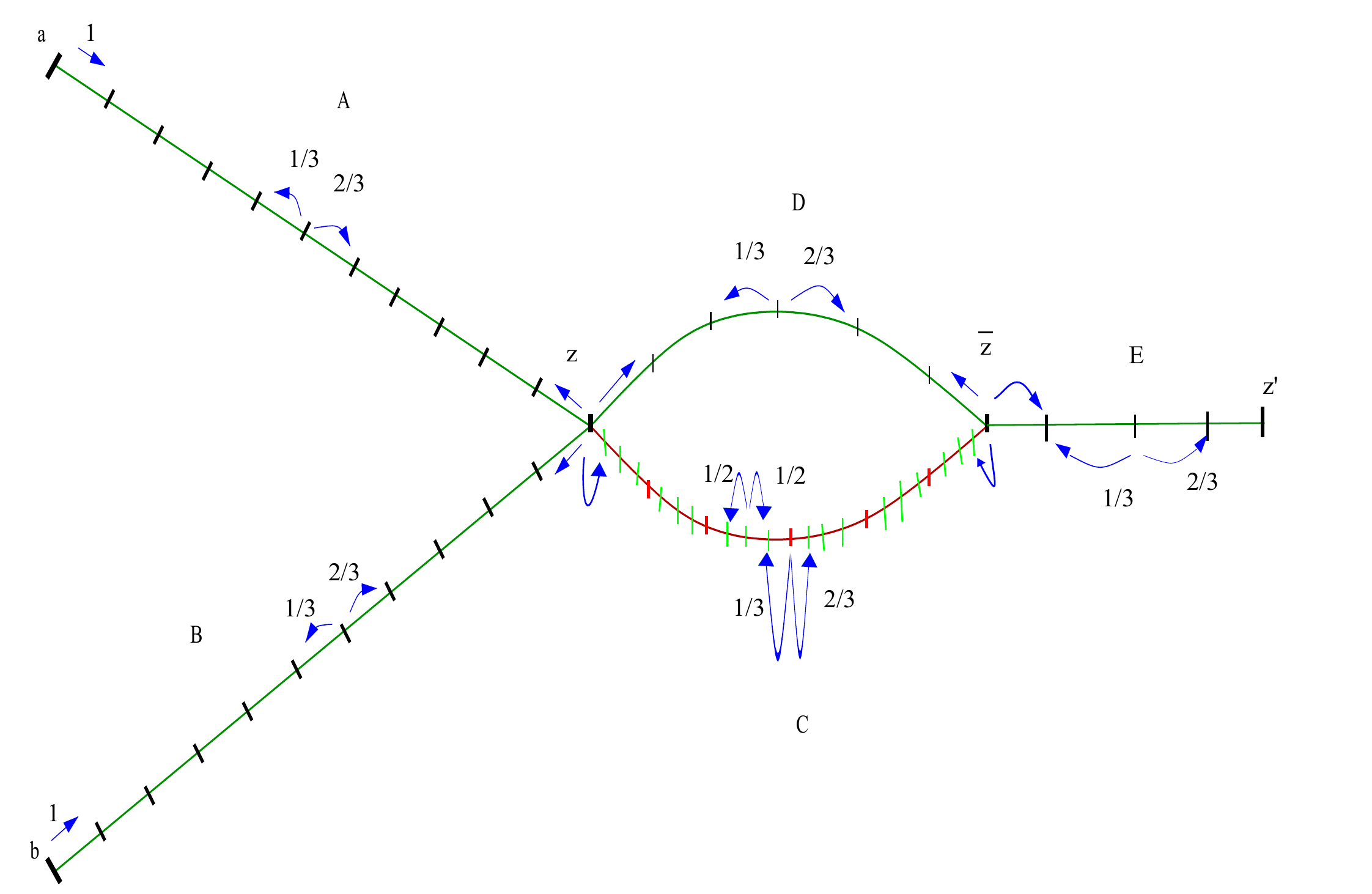}
\end{center}
\caption{\label{fig:ex3} 
A schematic representation of the transition probabilities for the chain from Theorem \ref{thm: ctslazy}. The sets $A$ and $B$ are of size $n$ and the sets $D$ and $E$ are of size $\delta n/2$. Along the branch $C$ there are $\delta n/2$ red points (with transitions probabilities $1/3$ to the left and $2/3$ to the right). Between each 2 red points there are $s-1$ yellow points, from each of which the chain moves to one of its 2 neighbors with equal probability. Also from $z$ and $\bar z$ the transitions are of equal probability to each of their neighbors. For a motivation behind the construction see the discussion following Remark \ref{rem:34} and the one preceding the specification of the edge weights of the chain.}
\end{figure}

Note that the restriction of the chain to $D$ is a birth and death chain with a bias towards $\bar z$ and an average speed of $1/3 $ ($1/6$ for the lazy chain), while its restriction to $C$ can be described as follows: first take the same birth and death chain as $D$ and then ``stretch" each edge $(c_{i+1},c_{i})$ by a factor $s$ by replacing it by a path of $s$ edges, $(c_{i,s},c_{i,s-1})$, \ldots, $(c_{i+1,1},c_{i,0})$, of the same weight as  $(c_{i+1},c_{i})$. It is not hard to see that this results in an average speed of $1/(3s^2)$ along $C$ towards $\bar z$ (for the lazy chain the speed is $1/(6s^2)$).

Note that started from $a $ the chain may reach $\bar z $ (and thus also $z'$) either through the branch $C$ or $D$ (i.e.~the last state to be visited prior to $T_{\bar z}$ may be either in $D$ or in $C$). Started from $a$, conditioned on taking the branch $C$ (in the above sense), the hitting times of $\bar z $ and $z'$ (for the continuous-time chain) are concentrated around $3(1+ s^2\delta /2)n$ and $3(1+\frac{\delta}{2}(1+ s^2))n $, resp., while conditioned on taking the branch $D$, they are concentrated around $3(1+\frac{\delta}{2})n$ and $3(1+\delta )n $, resp.. From this, we get that there exists some constant\footnote{Indeed $c$ is bounded away from 0, since  $s$ is fixed.} $c=c(s)=\Theta(1/s) $ such that for all sufficiently large $n$
\begin{equation}
\label{yinC1}
\forall y \in C \cup D \cup E \cup \{\bar z ,z'\}  , \quad  \h_a[T_{y} \le \left( 3(1+\frac{\delta}{2} s^2)+\eps \right) n ] \ge c.
\end{equation} 
\begin{equation}
\label{yinC2'}
\h_a[T_{z'} \le \left( 3(1+\frac{\delta}{2}(1+ s^2))- \eps  \right)n] \le 1-c.
\end{equation} 
\begin{equation}
\label{yinC3}
\forall y \in C \cup D, \quad  \mathbb{P}[T_{\bar z,y}^{a} \le \left( 6(1+\frac{\delta}{2}(1+ s^2))+\eps \right) n ] = 1-o(1).
\end{equation} 
\begin{equation}
\label{yinC4}
\forall y \in E \cup \{\bar z ,z'\}, \quad \Pr_a[T_{y} \le \left( 6(1+\frac{\delta}{2}(1+ s^2))+\eps \right) n ] = 1-o(1).
\end{equation}

Note that (since $s$ is fixed)
\begin{equation}
\label{piz}
\pi(z)=\Theta(2^{- \delta n}).
\end{equation}

 We now argue that
\begin{equation}
\label{eq: potentialmin}
\begin{split}
& \limsup_{n \to \infty} \sup_t |d_{\mathrm{sep,c}}^{(n)}(t)-\max \{ 0, 1- \min_{y \in \{b,z' \}}H_t(a,y)/\pi(y)\}  |=0,
\\ & \limsup_{n \to \infty} \sup_t |d_{\mathrm{sep,L}}^{(n)}(t)-\max \{ 0, 1- \min_{y \in  \{b, z' \}}P_{\mathrm{L}}^{t}(a,y)/\pi(y) ]\} |=0.
\end{split}
\end{equation} 
\begin{equation}
\label{eq: ab}
\begin{split}
& \lim_{t' \to \infty} \liminf_{n \to \infty} H_{\tau_{\del}+t'}(a,b)/\pi(b) = \infty, \quad \lim_{t' \to \infty} \limsup_{n \to \infty} H_{\tau_{\del}-t'}(a,b)/\pi(b) =0,
\\ & \lim_{t' \to \infty} \liminf_{n \to \infty} P_{\mathrm{L}}^{t_{\del}+t'}(a,b)/\pi(b) = \infty, \quad \lim_{t' \to \infty} \limsup_{n \to \infty}P_{\mathrm{L}}^{t_{\del}-t'}(a,b)/\pi(b) =0.
\end{split}
\end{equation}
These equations are essentially borrowed from the analysis of Examples 3 and 5 in \cite{cf:Hermon}. For the sake of completeness we present a sketch of their proofs (all of the missing steps can be found at \cite{cf:Hermon}). We now prove \eqref{eq: potentialmin}. We only prove the first line of \eqref{eq: potentialmin} as the second line is proved in a similar fashion. The proof of \eqref{eq: ab} is contained within the analysis of Case 1 below. 

By \eqref{eq: septhroughz4} with $(x,y)=(a,z')$  (and noting that $\reln = \Theta(1)= \pi(z') $ while, in the notation of \eqref{eq: septhroughz4}, $\lim_{n \to \infty} \max_t f_{z'}^{a}(t)=0 $ (cf.~\cite[Remark 3.9]{cf:Hermon}), where $z'=z'(n)$ and $f_{z'}^a=(f_{z'}^a)^{(n)} $)
\begin{equation}
\label{eq: atoz'}
|H_t(a,z')/\pi(z') - \h_a[T_{z'} \le t]|=o(1), 
\end{equation}
uniformly in $t$.
Denote $F:=C \cup D \cup E\cup \{z,\bar z,z' \} $. By symmetry between $A$ and $B$ it suffices to consider each of the following cases: (Case 1) $(x,y) \in A \times B $, (Case 2) $(x,y) \in A \times A $, (Case 3) $(x,y) \in A \times F $ and (Case 4) $(x,y) \in F\times F $ (where $(x,y)$ may depend on $n$). 

For Case 2, if $(x,y)=(a_i,a_j)$ we may assume w.l.o.g.\ that $i>j$ (since $H_t(x,y)/\pi(y)=H_t(y,x)/\pi(x)$). Using \eqref{eq: septhroughz3} we get that for all $t \ge 3n+n^{2/3} $ we have that $1-H_t(x,y)/\pi(y) \le \h_x[T_x \le y]=o(1)$ while by \eqref{eq: atoz'} we have that $H_{n(3+\delta/2)}(a,z')/\pi(z')=o(1)$ (hence we may neglect Case 2).\footnote{The exponent $2/3$ can be replaced by an arbitrary exponent in $(1/2,1)$.}

For Case 4  it is not hard to see that for $j(n):=n\frac{\delta}{2}(s^2+1)+n^{2/3}  $ we have
\[ \max \{ \mathbb{P}[\tau_{\bar z,y}^{x} \le j(n)],\mathbb{P}[\tau_{\bar z,x}^{y} \le j(n)]  ,\h_x[T_y \le j(n)], \h_y[T_x \le j(n)] \} =1-o(1) \] 
(uniformly in all $(x,y)$ as in Case 4). Thus by Lemma \ref{lem: pathsdecomposition}, for all $t \ge j(n)$ we have that $1-H_t(x,y)/\pi(y)=o(1) $, uniformly
in all $(x,y)$ as in Case 4.  

We now consider Case 3. For all $x \in A$, if $y \in C \cup D$ then by \eqref{eq: septhroughz1} we have that \[H_t(x,y)/\pi(y) \ge \max \{ \mathbb{P}[\tau_{\bar z,y}^x \le t] ,  \mathbb{P}[t-1 \le \tau_{\bar z}^{x,y} \le t] /[e \pi(z) ] \} \] and if $y \in E \cup \{z,\bar z ,z' \} $ we have \[ H_t(x,y)/\pi(y) \ge \h_x[T_y \le t] .\] By \eqref{eq: septhroughz4} (as in \eqref{eq: atoz'}) $|H_t(x,z')/\pi(z') - \h_x[T_{z'} \le t]|=o(1)$ uniformly for all $x \in A$ and $t \ge 0$. It is thus not hard to see that for each $x \in A$ the worst $y \in F$ (in the sense of minimizing $\min \{ H_t(x,y)/\pi(y),1 \} $ for all $t$, up to $o(1)$ additive terms)  is $y= z'$ and that for each fixed $t$ the worst $x \in A$ w.r.t.~$y=z'$ (at least up to $o(1)$ additive terms) is $x=a$.   

Finally, for Case 1 note that by \eqref{eq: septhroughz1} we have that \begin{equation}
\label{eq: t*xy} \frac{\mathbb{P}[t-1 \le \tau_z^{x,y} \le t]}{e\pi(z)} \le H_{t}(x,y)/\pi(y) \le \mathbb{P}[\tau_z^{x,y} \le t]/\pi(z).   \end{equation}

Using the fact that the law of $\tau_z^{x,y}  $ is a convolution of Exponential distributions and hence log-concave (cf.~the analysis of Example 5 in \cite[Section 6.5]{cf:Hermon}), it follows that for all $(x,y) \in A \times B$ we have that $H_{t}(x,y)/\pi(y)  $ exhibits the following behavior   around  \[t_{x,y}:=\inf \{t: \mathbb{P}[t-1 \le \tau_z^{x,y} \le t] \ge \pi(z)/4 \}.\] For every $\alpha \in (0,1)$ if $t \le (1-\alpha)\mathbb{E}[\tau_z^{x,y} ]$, there exists some $C_{\alpha}$ such that (uniformly in all $(x,y)$ as in Case 1)  \[\mathbb{P}[\tau_z^{x,y} \le t] \le C_{\alpha} e^{-1} \mathbb{P}[t-1 \le \tau_z^{x,y} \le t]  \] for some constant $C_{\alpha}>0$ (this uses the aforementioned log-concavity of $\tau_z^{x,y}$ (cf.~\cite[Section 6.5]{cf:Hermon}) and so by \eqref{eq: t*xy} (and the fact that $t_{x,y}< (1-\alpha)\mathbb{E}[\tau_z^{x,y} ]$ for some $\alpha>0$ for all sufficiently large $n$) there exist a constant $C>0$ such that 
\[\frac{\mathbb{P}[t-1 \le \tau_z^{x,y} \le t]}{e\pi(z)} \le H_{t}(x,y)/\pi(y) \le C \frac{\mathbb{P}[t-1 \le \tau_z^{x,y} \le t]}{e\pi(z)} .  \]
Again, using  log-concavity and the fact that $t_{x,y}$ by definition is in the large deviation regime of $\tau_z^{x,y} $, such that for some constants $c,C'>0$  for each fixed $k \in \Z$ we have that $e^{ck} \le \frac{\mathbb{P}[t_{x,y}+k-1 \le \tau_z^{x,y} \le t_{x,y}+k]}{\mathbb{P}[t_{x,y}-1 \le \tau_z^{x,y} \le t_{x,y}]} \le e^{C'k} $ for all sufficiently large $n$ (cf.~\cite[Section 6.5]{cf:Hermon}). Hence
 \[ \lim_{t' \to \infty} \liminf_{n \to \infty}\min_{(x,y) \in A \times B} H_{t_{x,y}+t'}(x,y)/\pi(y) = \infty, \] \[ \lim_{t' \to \infty} \limsup_{n \to \infty} \max_{(x,y) \in A \times B} H_{\max \{t_{x,y}-t',0\}}(x,y)/\pi(y) =0. \]
This in particular establishes the first line of \eqref{eq: ab} (the proof of the second line is analogous, where Exponential distributions above are replaced by Geometric distributions (see  \cite[Section 6.5]{cf:Hermon}). Moreover, the same reasoning as in \cite[Section 6.5]{cf:Hermon} yields that $H_{t}(x,y)/\pi(y)$ is increasing in $[0,t^{*}(x,y)]$, for some $t^{*}(x,y)$  satisfying $|t^{*}(x,y)-\mathbb{E}[\tau_z^{x,y}]| \le C_2 \sqrt{n} $, and that for all $t \ge t^{*}(x,y) $ we have that $ H_{t}(x,y)/\pi(y) \ge 1-o(1)  $ (uniformly in $(x,y) \in A \times B$ and $t \ge t^*(x,y)$). Finally, to conclude the analysis of Case 1 we note that $\max_{(x,y) \in A \times B}t^*(x,y)=t^*(a,b)$. 

\medskip

We are now in a position to conclude the proof. We first argue that the sequence of lazy chains exhibits separation cutoff around time $t_{\delta} $. Indeed,  by \eqref{eq: septhroughz1}-\eqref{eq: septhroughz3} in conjunction with \eqref{yinC3}-\eqref{yinC4} we have  
 $\max_{y \in  C \cup D \cup E \cup  \{z, \bar z, z' \} }\sup_{t \ge \left( 6(1+\frac{\delta}{2}(1+ s^2))+\eps \right) n} (1-\frac{P_{\mathrm{L}}^{ t}(a,y)}{\pi(y)} )=o(1)$. By \eqref{eq: deldel's} $\left( 6(1+\frac{\delta}{2}(1+ s^2))+\eps \right) n \le t_{\delta}-\epsilon n $ 
 and thus  by combining \eqref{eq: potentialmin} and \eqref{eq: ab},  we obtain that $\limsup_{n \to \infty} \sup_t |d_{\mathrm{sep,L}}^{(n)}(t)-\max(0, 1- P_{\mathrm{L}}^{t}(a,b)/\pi(b) ]) |=0$ and that $\max(0, 1- P_{\mathrm{L}}^{t}(a,b)/\pi(b) ]) $ exhibits a sharp transition around time $t_{\delta}$.

We now consider the continuous-time chain. By \eqref{eq: potentialmin}, \eqref{eq: ab}, \eqref{yinC1} combined with \eqref{eq: septhroughz1}, \eqref{eq: septhroughz3} and \eqref{eq: deldel's}  
 we get that $d_{\mathrm{sep},c}^{(n)}(\tau_{\del} + \eps n/2) $ is bounded away from 1 uniformly in $n$ (using the fact that the separation distance is non-increasing in time). Denote $m_n:=  3(1+\frac{\delta}{2}(1+ s^2))n - \eps n $. By \eqref{eq: deldel's}, $\tau_{\del}+\eps n < m_{n} $. Thus the proof is concluded by applying \eqref{eq: atoz'} with $t=m_n$, in order to deduce  by \eqref{yinC2'} that $$ d_{\mathrm{sep},c}^{(n)}(m_{n}) \ge 1-H_{  m_{n}}(a,z')/\pi(z') \ge \h_a[T_{z'} > m_n ]-o(1) \ge c-o(1). $$   \qed    

\subsection{Proof of Remark \ref{rem: lazypqintro}}
\label{sec: lazypq}
We now explain the necessary adaptations for the proof of the assertion of Remark \ref{rem: lazypqintro}.
The only adjustments are in the choices of $\delta$ and $s$ (and possibly, one has to contract some of the $s$-paths along the slow branch $C$ to a single edge in order to adjust the expected time it takes the walk to cross the slow branch). Fix $\alpha \in (0,1)$. Consider the $\alpha$-lazy version of the chain. Denote by $\Pr_x^{(\alpha)} $ the law of the $\alpha$-lazy version of the chain and by $T_{x'}^{x,y}(\alpha) $ the version of $T_{x'}^{x,y}$ corresponding to holding probability $\alpha$. Note that
under the aforementioned modifications to the chain (i.e.~adjusting $\del$ and $s$) the transition probabilities along the $A$ and $B$ segments are unaffected and so (for each fixed $\alpha$) the law of $T_{z}^{a,b}(\alpha) $ is also unaffected.
 
Let $\delta>0$ to be determined later. The aforementioned modifications can be made so that we still have that $\pi(z)=\Theta(2^{-\delta n})$. Denote the separation distance at time $t$ of the $n$th $\alpha$-lazy chain by  $d_{\mathrm{sep},\alpha}^{(n)}(t)$. 
Let $\kappa n $ be the expected hitting time of $z$ from $\bar z$ for the non-lazy chain, conditioned on taking the slow branch $C$. Similarly to the analysis in the proof of Theorem \ref{thm: ctslazy}, as long as $\delta$ is taken to be sufficiently small, and $\kappa$ is chosen in an appropriate manner\footnote{More precisely, this is the case whenever $\frac{3(1+\max( \delta ,\kappa))n }{1-\alpha}< t_{\delta}^{(\alpha)}-\epsilon n $ for some $\epsilon >0$, where $ t_{\delta}^{(\alpha)}$ is defined in the following paragraph. We leave this as an exercise.}, one can replace $d_{\mathrm{sep},\alpha}^{(n)}(t) $ by $\max \{ \Pr_a^{(\alpha)}(T_{z'} > t), 1-\frac{\Pr_a^{(\alpha)}(X_t=b)}{\pi (b)} \} \pm o(1) $.

 Let $t_{\delta}^{(\alpha)}:=\inf \{t: \mathbb{P}[T_{z}^{a,b}(\alpha) \le t ] \ge 2^{- \delta n} \}$. Similarly to the analysis in the proof of Theorem \ref{thm: ctslazy}, the quantity $1-\frac{\Pr_a^{(\alpha)}(X_t=b)}{\pi (b)}$ exhibits a sharp transition around $t_{\delta}^{(\alpha)}$. Using the aforementioned modifications we can ensure that $\Pr_a^{(\alpha)}(T_{z'} \le t) $  exhibits two jumps, one around $\frac{3(1+\delta)n }{1-\alpha}$ and a second jump around $s^{(\alpha)}:= \frac{3(1+\delta /2+\kappa)n }{1-\alpha}$. It follows that the sequence of $\alpha$-lazy chains exhibits separation cutoff iff $t_{\delta}^{(\alpha)}\ge (1-o(1)) s^{(\alpha)}  $.  

 Using the analysis from \cite[Example 3 and Lemma 5.1]{cf:Hermon} one can derive the following formula for $\Psi_{\alpha} $, the large deviation rate functions of $T_{z}^{a,b}(\alpha) $, which is given by the following Legendre transform (the exact formula shall  be used only to make a certain comment at the end of the analysis)
 \begin{equation}
  \Psi_{\alpha}(r):=\sup_{\gl\in (-\infty,\infty)} \left[ \gl r- \log \tf_{\alpha}(\gl) \right],
 \end{equation}
 where for $\Delta_{\alpha}(\lambda):=(e^{-\gl}-\alpha)^2-\frac{4(1-\alpha)}{3} $ and $\gl_{\alpha}$, the smaller solution to $\Delta_{\alpha}(\lambda)=0 $, we have that
 \begin{equation*}
  \tf_{\alpha}(\gl):=\begin{cases}
             \infty   & \text{if } \gl >\gl_{\alpha}),\\
           \frac{3}{2(1-\alpha)} [(e^{-\gl}-\alpha)- \sqrt{\Delta_{\alpha}(\lambda)}] \quad &\text{ if } \gl \le \gl_{\alpha}).  
            \end{cases}
 \end{equation*}
Fix $p \neq q \in (0,1) $. Using the fact that the Legendre transform of a strictly convex smooth  function is itself smooth and that the restriction of the Legendre transform to this class of functions is invertible, it is not hard to verify that for every $\epsilon \in (0,1/100)$ there is some $r \in (3-\epsilon,3) $ such that $\Psi_{p}(\frac{r}{1-p}) \neq \Psi_{q}(\frac{r}{1-q})  $.\footnote{The choice of the constant $1/100$ is arbitrary and is made in order to ensure that $3-r$ is small.} This implies that for all  $p \neq q \in (0,1) $, we can find $\delta=\delta_{p,q} \in (0,1/10) $ such that for some $\delta'=\delta_{p,q}' >0$ either (Case 1:) $(1-p)t_{\delta}^{(p)}-(1-q)t_{\delta}^{(q)} \ge \delta'n(1 \pm o(1) )$, for all $n$ or (Case 2:)  $(1-q)t_{\delta}^{(q)}-(1-p)t_{\delta}^{(p)} \ge \delta'n(1 \pm o(1) )$ for all $n$ (namely, fixing such $r$ we can pick $\delta$ to satisfy  $2^{-n \delta}=e^{-2n\Psi_{p}(\frac{r}{1-p}) } $; Case 1 corresponds to $\Psi_{p}(\frac{r}{1-p}) > \Psi_{q}(\frac{r}{1-q})$ and Case 2  to  $\Psi_{p}(\frac{r}{1-p}) < \Psi_{q}(\frac{r}{1-q})$).  This implies that we can tune $\kappa$ such that in Case 1 we have $t_{\delta}^{(p)}>s^{(p)} $ and $\limsup_{n \to \infty} t_{\delta}^{(q)}/s^{(q)}<1 $ and in Case 2 we have  $t_{\delta}^{(q)}>s^{(q)} $ and $\limsup_{n \to \infty} t_{\delta}^{(p)}/s^{(p)}<1 $. Thus in the first case the $p$-lazy chains exhibit separation cutoff while the $q$-lazy chains do not, and in the second case the $q$-lazy chains exhibit separation cutoff while the $p$-lazy chains do not. Note that in the above argument the formulas for $ \Psi_{p}$ and $\Psi_{q}$ played no role. However, they can be used to distinguish between Cases 1 and 2 for every fixed $p \neq q$.   

\section{ Proof of Theorem \ref{thm: sensitive2}}
\subsection{Preliminaries}
Before proving Theorem \ref{thm: sensitive2} we make several general comments regarding a principle which shall be utilized below repeatedly. We summarize a few different variations of this principle in Fact \ref{fact: SRWonZ} (whose proof is deferred to the appendix \S~\ref{s:F41}).
\medskip

Let $\cT=(V,E)$ be an infinite binary tree rooted at $o$ (in practice, we shall work with finite trees; however, it is not hard to show that the ``boundary effect" coming from the finiteness of the trees is negligible for our poruses). For any vertex $u $ we distinguish its two children by \emph{left} and \emph{right} child. We denote the collection of all left (resp.~right) children in $\cT$ by $L$ (resp.~$R$). Denote by $\cL_n$ the collection of vertices whose distance from $o$ is $n$. For any vertex $u  $ let $L(u)$, (resp.~$R(u)$) be the number of left (resp.~right) children along the path from $u$ to the root. Denote $g(u)=L(u)-R(u)$.
Let $\tau_{k}:=\sup \{t:X_t \in \cL_k \}$.
\begin{fact}
\label{fact: SRWonZ}
\begin{itemize}
\item[(1 a)]Starting from $o$, for each fixed $n$ we have that $g(X_{T_{\cL_n}})$ is distributed like $S_{n}$, where $(S_k)_{k \in \Z_+ }$ is a SRW on $\Z$, started at the origin. Moreover, $(g(X_{\tau_k}))_{k \ge 0}$ is distributed like $(S_k)_{k \in \Z_+ }$.
\item[(1 b)]
Let $n \in \N$ and $D \subset \cL_n$. Let $\mathcal{D}$ be the event that the walk visits the set $D$ at least once. There exists an absolute constance $c>0$ such that
\begin{equation}
\label{eq: hitD1}
c\le \Pr_o[X_{\tau_n} \in D]/\Pr_o[\cD] \le 1 \text{ and } c\le \Pr_o[X_{T_{\cL_n}} \in D]/\Pr_o[\cD] \le 1.
\end{equation}

\medskip

\item[(2 a)] Fix some $ \epsilon >0 $. Consider the network obtained by increasing the edge weight between every $u$ and $v$ such that $v$ is a left child of $u$ to $1+\epsilon$. Then, $(g(X_{\tau_k}))_{k \ge 0}$ is distributed like a biased nearest-neighbor random walk on $\Z$, $(\tilde S_k)_{k \in \Z_+}$, satisfying $\Pr[\tilde S_1=1]=1-\Pr[\tilde S_1=-1]=\frac{\sqrt{1+\epsilon}}{1+\sqrt{1+\epsilon}} =\frac{1}{2}+\epsilon/4-O(\epsilon^2)>1/2$.

\medskip

\item[(2 b)]
Let $n \in \N$ and $D \subset \cL_n$. Let $\mathcal{D}$ be the event that the walk visits the set $D$ at least once. There exists an absolute constance $c$ (which is independent also of $\epsilon$)  such that also in the perturbed network we have that \eqref{eq: hitD1} holds.
\end{itemize}
\end{fact}

\begin{fact}[Local CLT] Let $(S_k)_{k \in \Z_+}$ be simple random walk (SRW) on $\Z$. Then there exists an absolute constant $C>1$ such that for all $n \ge 1$ and $1 \le m \le \lceil n^{1/4} \rceil $
\begin{equation}
\label{eq: LCLT}
1/C \le \Pr_0[S_n \ge m \sqrt{n} ]m e^{m^2/2} \le C .
\end{equation}
\end{fact}
%
\subsection{Proof of part (a) of Theorem \ref{thm: sensitive2}}
Fix some (large) $k \in \N$. We suppress the dependence on $k$ from the notation (below, $o(\cdot), O(\cdot),\Theta(\cdot)$ and $\Omega(\cdot) $ are taken w.r.t.~$k$). Denote $s:=k^3$. We now construct a graph $G=(V,E)$ in three steps (see Figure \ref{fig:thm3}).
\begin{figure}[!h]
\begin{center}
\includegraphics[height=6cm,width=9cm]{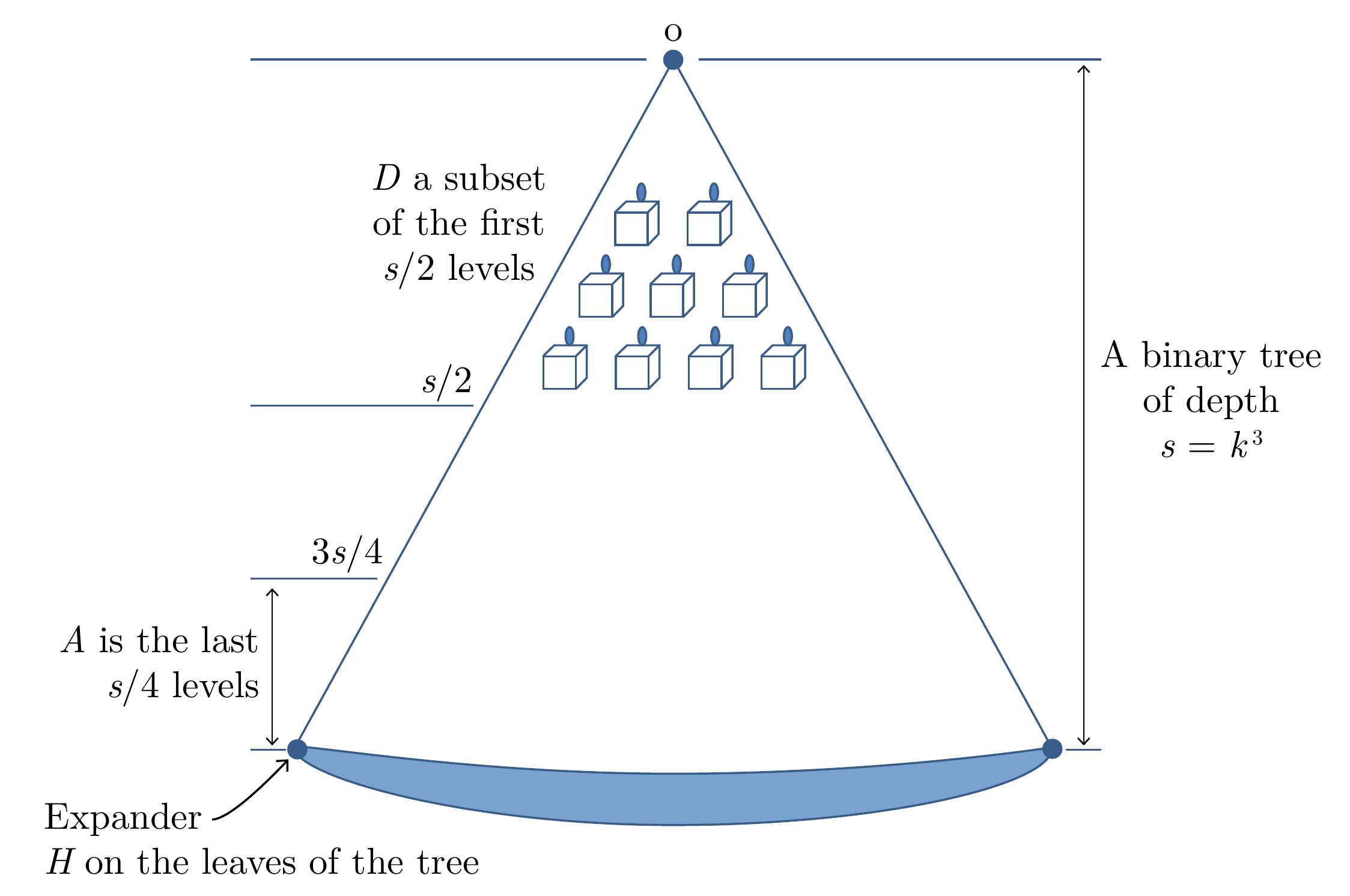}
\end{center}
\caption{\label{fig:thm3} This is schematic representation of the graph $G$ from part (a) of Theorem \ref{thm: sensitive2}. We start with a binary tree rooted at $o$ of depth $s:=k^3$. We then connect its leafs using an expander. We define a set $D$, contained in the first $s/2$ levels of the tree, which is the collection of ``unbalanced" vertices in the sense that the number of left/right turns from them to the root violates the Law of Iterated Logarithm in some strong sense.  Finally, $d \in D$ is decorated by a $k \times k \times k $ torus, represented by a square.}
\label{fig:thm3}
\end{figure}
\begin{itemize}
\item[Step 1:]
Start with a binary tree $\cT=(V(\cT),E(\cT)) $ of depth $s$ with root $o$.
\item[Step 2:] Form an expander on the leaves of $\cT$: Let $H=(V',E')$ be an arbitrary expander with $|V'|=2^{s}$. Identify the leaves of $\cT$ with the set $V'$ and connect by an edge every pair of leaves $u,v $ such that $\{u,v \} \in E'$.
\item[Step 3:] We pick a set $D \subset V$ strategically as follows and decorate each of its vertices by a 3D torus of side length $k$ (and so of size $s$):

Denote the collection of vertices belonging to the $i$-th level of $\cT$ by $\cL_i(\cT)$. For any vertex $u $ which is not a leaf of $\cT$, we distinguish its two children by \emph{left} and \emph{right} child. We denote the collection of all left (resp.~right) children in $\cT$ by $L$ (resp.~$R$). Fix some large integer $C$ to be determined later. We note that one can set $C=1$, but the analysis is somewhat smoother by taking $C$ to be large. Including the constant $C$ in the construction shall benefit us in the proof of part (b). We denote by $D_{i}$ the collection of all vertices $u$ belonging to $\cL_{Ci}(\cT) $ such that if $\gamma_u=(v_0=u,v_1,\ldots,v_{Ci}=o)$ is the path from $u$ to $o$ in $\cT$, then for all $1\le j \le i$ $$|\{v_{\ell}: 0 \le \ell \le Cj  \} \cap L|-|\{v_{\ell}: 0 \le \ell \le Cj  \} \cap R| \ge 3\sqrt{Cj\log \log (Cj)}.$$
Crucially, above the ``base point" $v_0$ of $\gamma_u$ was taken to be $u$ itself (rather than $o$).

Let $D_0:=\{o\}$. Denote $D:=\bigcup_{i=0}^{\lfloor s/2C\rfloor} D_i$. We decorate each $v \in D$ by a 3D $k \times k \times k  $ torus. That is, for each $v \in D$ we attach to $v$ a three dimensional torus, $W_v$, of side length $k$, having $v$ as one of its vertices, while the rest of its vertices are disjoint from $\cT$ (where $W_v$ and $W_u$ are taken to be disjoint if $v \neq u$). Call the resulting graph $G=(V,E)$.
\end{itemize}
We argue that
\begin{equation}
\label{eq: k4}
t_{\mathrm{mix}}(G)=\Theta (s) \text{ and also } \rel(G)=\Theta (s).
\end{equation}
By Lemma \ref{lem: relaxation} $\rel(G)=\Omega (s)$.
We now explain why indeed $t_{\mathrm{mix}}(G)=\Theta (s)$ (by \eqref{eq: t_relintro} this implies that $\rel(G)=\Theta(s)$). Let $A:=\bigcup_{j \ge 3s/4}\cL_j(\cT) $. The set $A$ is sufficiently far from the set $D$ so that starting from any $u \in A$ the walk mixes in $\Theta(s)$ steps (as if the vertices in $D$ were not decorated by tori). That is, there exists an absolute constant $C_1 \in \N $ such that for every $u \in A$,
\begin{equation}
\label{eq: mixingcenter1}
\|\Pr_{u}^{ C_1s} - \pi \|_{\TV}=o(1).
\end{equation}
This can be deduced formally using Proposition \ref{prop: L2contractiononasubchain}. Proposition \ref{prop: L2contractiononasubchain} applies  because for all $a \in A$, by the tree structure we have that $\Pr_{a}[T_{D} \le C_1s] $ can be bounded from above by the probability that a LSRW on a binary tree of depth $\lfloor s/4 \rfloor$, started from some leaf, reaches the root by time $C_1s$ (which occurs with probability of at most $C'e^{-cs}$). Consequently, by Proposition \ref{prop: L2contractiononasubchain}, there exists some absolute constant $C_2'$ such that for every $t \ge 0$
\begin{equation}
\label{eq: mixhit2}
\begin{split}
& \max_{u \in V} \Pr_u[T_A>t] \le d(t)+(1-\pi(A))=d(t)+o(1) .
\\ & \max_{u \in V} \Pr_u[T_A>t] \ge d(t+C'_{2}s)-o(1).
\end{split}
\end{equation}
Let $T_1$ (resp.~$T_2$) be the total amount of time, prior to time $T_A$, that the walk spends at $\bigcup_{v \in D}W_v$ (resp.~$V \setminus \bigcup_{v \in D}W_v $). Let $T_3$ be the total number of times the set $D$  was visited prior to time $T_A$ by crossing some edge belonging to $\cT$. That is, $$T_3:=|\{1 \le t<T_A:X_{t} \in D, X_{t-1} \in V(\cT) \}|.$$ We argue that there exist absolute constants  $C_2,C_{3},\beta>0$ such that
\begin{itemize}
\item[(1)]
$$\lim_{m \to \infty} \max_{u \in V} \Pr_u[T_3>m]=0. $$
\item[(2)]
$$\max_{u \in V} \Pr_u[T_2> 5s]=o(1).$$
\item[(3)] For every $m,r \in \N$ we have that
$$  \max_{u \in V} \Pr_u[T_1 >C_{2}mrs \mid T_3=m ] \le C_3 e^{-\beta mr}. $$
\end{itemize}

Combining (1)-(3) with (\ref{eq: mixhit2}) concludes the proof of the fact that $t_{\mathrm{mix}}(G)=\Theta (s)$.

To see why (1) holds, use part (1 a) of Fact \ref{fact: SRWonZ}, in conjunction with the law of the iterated logarithm and the fact that the distribution of the number of visits by time $T_A$  to each $\cL_i(\cT)$, denoted by $N_i$, has an exponential tail (along with the fact that $\Cov (N_i,N_{i+j}) $ decays exponentially in $j$ for all $i,j$). We leave the details as an exercise.

For (2), note that the distance from the root of a LSRW on a binary tree behaves like a biased nearest neighbor walk whose average speed is $1/6$.

It is not hard to show that there exist some $C_4,C_5,\beta'>0 $ such that for all  $v \in D $
\begin{equation}
\label{eq: delay1}
\max_{u \in W_v} \Pr_{u}[T_{V \setminus W_v}> ms] \le C_{4}e^{-\beta'm}, \text{ for every }m \in \N.
\end{equation}
\begin{equation}
\label{eq: delay2}
\forall u \in W_v, \quad  s/2 \le \mathbb{E}_u[T_{V \setminus W_v}] \le C_5s.
\end{equation}
Thus (3) is obtained as a large deviation estimate.

\medskip

The bounded perturbation from the assertion of part (a) of Theorem \ref{thm: sensitive2} is obtained by increasing the edge weight of the edge between any $v \in \bigcup_{0 \le i \le s/2} \cL_{i}(\cT) $ and its left child to $1+\epsilon$, for some constant $\epsilon>0$.

\medskip
It is easy to show that (\ref{eq: mixhit2}) remains valid also in the perturbed network and that (by symmetry) the maximum (in the LHSs of (\ref{eq: mixhit2})) can (still) be taken over the set $W_o$. To distinguish between the LSRW on $G$ and the lazy random walk on the perturbed network we adopt the convention that when referring to the perturbed network we write $\Pr_u'$ and $\mathbb{E}_u'$, instead of $\Pr_u$ and $\mathbb{E}_u$.

\medskip

Let $T_1,T_2,T_3$ be as above. Using part (2) of Fact \ref{fact: SRWonZ} it is not hard to verify that for any $u \in W_o$ we have that $\mathbb{E}'_u[T_3] \ge c_{6}(\eps)s$ and  $\Var'_{u}[T_3] \le C_7(\eps)s$ for some constants $c_6(\eps),C_7(\eps)>0$, depending on $\epsilon$. The fact that $\mathbb{E}'_u[T_3] \ge c_{6}(\eps)s$ (for some $c_6(\eps)$) is clear from part (2 a) of Fact \ref{fact: SRWonZ}. To see that  $\Var'_{u}[T_3] \le C_7(\eps)s$ use part (2) of Fact \ref{fact: SRWonZ} to deduce that the correlation between the contribution to $T_3$ from vertices belonging to $D_i$ and $D_{i+j}$, respectively, decays exponentially in $j$.

Using these estimates, we now show that the order of the mixing time of the walk on the perturbed network is $\Theta (s^{2})$ and that  a sequence of walks on such perturbed networks with $k \to \infty$ exhibits cutoff. 

\medskip

Using (\ref{eq: delay1})-(\ref{eq: delay2}) it is not hard to show that for any $u \in W_o$, we have that $\mathbb{E}'_u[T_1] \ge c_{8}s^{2}$ and that  $\Var'_{u}[T_1] \le C_9 s^{3}=o((\mathbb{E}'_u[T_1])^2)$, for some constants $c_8,C_9>0$ depending on $\epsilon$. Moreover, for every $u \in W_o$, $ \mathbb{E}'_u[T_o]\le C_{5}s$. Consequently, it follows from Chebyshev's inequality that starting from every $u \in W_o$ we have that $T_A$ is concentrated around $\mathbb{E}'_o[T_1]=\Theta (s^{2})$. By (\ref{eq: mixhit2}) there exist some constants $c_{\epsilon},c'_{\epsilon}>0$ such that the mixing time of the walk on the perturbed network is $c'_{\epsilon}s^{2}(1 \pm o(1) )$ and thus the mixing time increased by a factor of $c_{\epsilon}\log |V|$.
Moreover,  a sequence of walks on the perturbed networks with $k$ tending to infinity exhibits cutoff (around time $\mathbb{E}'_o[T_1]$). \qed

\begin{remark}
The choice of a 3D torus was made to emphasize the similarity of the construction to the one from \cite{cf:Ding}; In fact, we could have used any bounded degree graph of size $\Theta(\log |V'|)=\Theta(s)$ which satisfies $\max_{x,y}\mathbb{E}_{x}[T_y] \le C s$.

\end{remark}

\subsection{Proof of part (b) of Theorem \ref{thm: sensitive2}}
We present two different constructions.

\textbf{First construction:} The first construction is obtained from the construction of the proof of part (a) by replacing the constant $C$ (from step 3) by some $C(k)$ tending to infinity as $k \to \infty $ such that $C(k)=o(s)$ (where as above $s=k^3$). Call the obtained network $G=(V,E)$. Clearly (\ref{eq: k4}) remains valid. This follows from the fact that \eqref{eq: mixhit2} remains valid, and that for all $t>0$ it is still true that $\max_{u \in V} \Pr_{u}[T_A>t]=\max_{u \in W_o} \Pr_{u}[T_A>t]$.

\medskip

Consider a perturbation of the same collection of edges which were perturbed in the proof of part (a), only that now we increase their weights to $1+K \sqrt{ \frac{\log \log (C(k)) }{C(k)}}$, where $K>4$ is some sufficiently large absolute constant to be determined shortly. Similar reasoning as in the proof of part (a) (using part (2 a) of Fact \ref{fact: SRWonZ}, together with the law of the iterated logarithm, here for a biased random walk on $\Z$, with a fixed bias) shows that if $K$ is sufficiently large, this perturbation increases the order of the mixing time by a factor of order $s/C(k)=\Theta( (\log |V|)/C(k))$ and that a sequence of random walks on the perturbed networks with $k \to \infty$ exhibits a cutoff. By taking $C(k)$ to tend to infinity arbitrarily slowly we can  increase the mixing time by any factor $f_k \to \infty $ such that $f_k=o(\log |V|)$.

\medskip

\textbf{Second construction:} We now present the second construction. Take sequences of integers $k_n,r_n,\ell_n,m_{n} $ tending to infinity as $n \to \infty$ such that $s_{n}=k_n^{3}=\ell_n r_n$, $m_n \le \lceil \ell_n^{1/4} \rceil $ and $r_n=\Theta \left(m_{n}e^{m_n^2/2} \right) $. The first two steps of the construction are taken as in the proof of part (a) with $k_n$ in the role of $k$. We modify step 3 by changing the definition of the set $D$ as follows.

Step 3': Set $D_0:=\{o\}$. For every $1 \le i \le s_{n}/(2\ell_{n})=r_{n}/2$, we set $D_i$ to be the collection of all vertices $u \in \cL_{i\ell_{n}}(\cT)$ such that if $(v_0=u,v_1,\ldots,v_{\ell_{n}})$ is the path in $\cT$ between $u$ and its $\ell_n$-th ancestor $v_{\ell_{n}}\in \cL_{(i-1)\ell_{n}}(\cT)$ then $$|\{v_{j}: 0 \le j \le \ell_{n}   \} \cap L|-|\{v_{j}: 0 \le j \le \ell_{n}  \} \cap R| \ge \left\lceil m_n \sqrt{\ell_n} \right\rceil.$$
By the local CLT (\ref{eq: LCLT}) and our choice of $m_n$ and $r_n$ we have that $ |D_1|2^{-\ell_{n}} = \Theta (r_{n}^{-1})$. Consequently, for every  $i \le r_n/2-1 $, for all $u \in \cL_{(i-1)\ell_n}(\cT)$ we have that
 \begin{equation}
 \label{eq: rn}
 \Pr_u[T_{D_{i}}=T_{\cL_{i\ell_n}(\cT)}]=\Theta (r_n^{-1}).
 \end{equation}
 We define $D:=\bigcup_{0 \le i \le r_{n}/2}D_i$. As before, we decorate each $v \in D$ by a 3D $k_{n} \times k_{n} \times k_{n}$ torus, $W_v$. Call the obtained graph $G_{n}=(V_{n},E_{n})$.

As before, let $A=A_{n}:=\bigcup_{j \ge 3s_{n}/4}\cL_j(\cT) $ and define $T_i$ in an analogous manner to the way they were defined in the proof of part (a) ($i=1,2,3$). By \eqref{eq: rn} we have that $\mathbb{E}[T_3]=\Theta(1) $ (using similar reasoning as in part (1 b) of Fact \ref{fact: SRWonZ}). Using this fact it is not hard to verify that (\ref{eq: k4}) remains valid also here (with $G_n$ and $s_n$ in the roles of $G$ and $s$). This follows from the fact that \eqref{eq: mixhit2} remains valid, and that for all $t>0$ it is still true that $\max_{u \in V_n} \Pr_{u}[T_A>t]=\max_{u \in W_o} \Pr_{u}[T_A>t]$.

\medskip

We now describe the perturbation of the edge weights described in part (b) of the theorem. Let $\delta_n=K
m_n/\sqrt{\ell_n}$, for some sufficiently large absolute constant $K$ to be determined later. Perturb the same edges as before by increasing their weights to $1+\delta_n$.

\medskip

Much as before, by symmetry, it suffices to consider the case that the initial state of the walk on the perturbed network belongs to $W_o$. Since $\delta_n \ell_n=Km_n \sqrt{\ell_n} $, if $K$ is taken to be sufficiently large , then by part (2) of Fact \ref{fact: SRWonZ}, for every  $i \le r_n/2-1 $ and all $u \in \cL_{(i-1)\ell_n}(\cT)$ we have that $\Pr_u'[T_{D_{i}}=T_{\cL_{i\ell_n}(\cT)}]=1-o(1)$.
This implies that if $K$ is taken to be sufficiently large, then in the perturbed network on $G_n$, starting from any $u \in W_o$, we have that $T_3$ is concentrated around some $t_n=\Theta (r_n)$. Consequently, $T_A$ is concentrated around some $t'_n=\Theta( r_{n}k_n^3)$, which, as before, implies that the walk on the perturbed network exhibits cutoff around time $t'_n$. In particular, the order of the mixing-times increased by a factor of $\Theta(r_n)$. Setting $m_n=\lceil \ell_n^{1/4}\rceil $, we get that $r_n=\Theta(\log |V_n|/\sqrt{\log \log |V_n|}) $. Note however that by taking $m_n$ to tend to infinity arbitrarily slowly we get that $s_n/\ell_n$ tends to infinity arbitrarily slowly and thus so does $\delta_n \sqrt{t_{\mathrm{mix}}(G_n)} $. 

\subsection{Proof of part (c) of Theorem \ref{thm: sensitive2}} Consider the graph $G$ from part (a). Now, stretch all of its edges by a factor 3. The tree $\cT $ is replaced by a tree $\cT'$ with stretched edges. However, we think of $\cT $ of as being contained  in $\cT'$, with each pair of neighbors in $\cT$ being separated by a path of length $3$ in $\cT'$. Denote the obtained graph by $G'$. It is not hard to see that the mixing time of $G'$ can differ from that of $G$ only by a constant factor (the expectation of the hitting time of the leaf set of $\cT$ is delayed by a factor of $3^2$). Now, for each path of length 3, $(u,w,w',v)$ connecting some vertex $u$ of $\cT$ and its left child $v$, lump together the pair of its internal vertices $w,w'$. This has the same effect as replacing the path $(u,w,w',v)$ by a path  $(u,z_{u,v},v)$ with a self loop of weight 2 at $z_{u,v}$.     It is easy to see that this results in a bias towards the left children of $\cT$ (the bias is the same as when we remove the self-loops at the $z_{u,v}$'s; After removing the self-loops, a standard network reduction, similar to the one in the proof of Fact \ref{fact: SRWonZ}, can be used to establish the existence of the aforementioned bias). The analysis can be concluded in a similar manner to the analysis of the perturbed network in part (a).
\qed

\begin{remark}
It is possible to modify the examples from Theorem \ref{thm: sensitive2} so that they also satisfy $\mathrm{girth}(G_n)=\Theta(t_{\mathrm{mix}}(G_{n}))$. Namely, parts (1)-(2) of the construction can be replaced by starting with some regular expander of logarithmic girth (\cite{cf:Moshe,cf:Xpq}) as the base graph, like in the construction of  Theorem \ref{thm sensitivity}. Then, instead of decorating the set $D$ by tori of size $\Theta(\log |V_n|)$, we could decorate them by binary trees of that size.
\end{remark}
\begin{remark}
Both Ding and Peres' example \cite{cf:Ding} and the example from Theorem \ref{thm sensitivity} satisfy the product condition.
It is thus natural to ask whether any sequence of bounded degree graphs satisfying $\rel(G_n)=\Theta (t_{\mathrm{mix}}(G_{n})) $ must be robust (equivalently, to ask whether the condition $\rel(G_n)=\Theta (t_{\mathrm{mix}}(G_{n}))$ is robust). Theorem \ref{thm: sensitive2} demonstrates that in fact the condition $\rel(G_n)=\Theta (t_{\mathrm{mix}}(G_{n}))$ may be $o(1)$-sensitive.
\end{remark}
\begin{remark}
\label{rem:NBRW}
A non-backtracking random walk on a simple graph $G=(V,E)$ (NBRW) evolves as follows. When at vertex $u$ at some time $t$, if at time $t-1$ the walk was at vertex $v$, the next position of the NBRW is chosen from the uniform distribution over $\{x \in V \setminus \{v\} :\{u,x\} \in E \}$.  We may consider the lazy version of a NBRW. 

A small variation of the example from part (a) of Theorem \ref{thm: sensitive2} shows that for a graph $G$ of bounded degree (with no degree 1 vertices) the total variation mixing time of the lazy NBRW may be larger than that of the LSRW by a factor of $\Theta (\log |V|)$. Namely, one can stretch each edge between a vertex in $\cT$ and its right child by a factor 2. As in part (c) of Theorem \ref{thm: sensitive2}, we think of $\cT$ as being contained in the modified tree. One can then define $D_{i}$ to be the collection of all vertices $u$ belonging to $\cL_{Ci}(\cT) $ such that if $\gamma_u=(v_0=u,v_1,\ldots,v_{Ci}=o)$ is the path from $u$ to $o$ in $\cT$, then for all $1\le j \le i$ $$|\{v_{\ell}: 0 \le \ell \le Cj  \} \cap L|-|\{v_{\ell}: 0 \le \ell \le Cj  \} \cap R| \le \epsilon j, $$
for some small $\epsilon>0$ and, as before, set $D:=\bigcup_{i=0}^{\lfloor s/2C\rfloor} D_i$. Finally, as before, decorate each $d \in D$ by a $3D$ torus of side length $k$ and connect the leafs of $\cT$ using an expander $H$. It is not hard to see that if $\epsilon$ is taken to be sufficiently small, then due to the bias towards the left children of $\cT$, resulting from stretching the ``right edges", $\mathrm{w.h.p.}$ the LSRW will visit only a constant number of tori before reaching the $3s/4 $ level of $\cT$. However,  the harmonic measure of the lazy NBRW is unaffected by the stretched edges (meaning that if $v$ and $v'$ are the children of $u$ in $\cT$, then also in the modified graph, when the lazy NBRW is at $u$, it has the same probability of reaching either $v$ or $v'$ before the other), which means that $\mathrm{w.h.p.}$ it will visit $\Theta(k^3)$ tori before reaching the $3s/4 $ level of $\cT$. Since also the lazy NBRW spends at average  $\Theta(k^3)$ steps at each torus, this means that the hitting time of the $3s/4 $ level of $\cT$ is $\Omega(k^6)$, which as before, implies that the mixing time of the lazy NBRW is $\Omega(k^6)$.  
\end{remark}
\section{Proof of Theorem \ref{thm sensitivity}}
Our construction is obtained by   stretching some of the edges of a {\em Ramanujan Cayley graph}.
Let $H$ be a group and $S \subset H$ be a finite symmetric (i.e.~$S=S^{-1}:=\{s^{-1}:s \in S\} $) set of generators of $H$ (i.e.~every $h \in H$ can be written as a finite product of the form $s_1s_2\cdots s_k $ where $s_i \in S$ for all $1 \le i \le k$). The \textbf{\emph{Cayley graph}} of $H$ w.r.t.~$S$ is defined to be the graph whose vertex set is $H$ and whose edge set is $\{\{h,hs \}:h \in H,s \in S  \}$.

\medskip

Let $G$ be a $d$-regular connected graph of size $n$. Denote the transition matrix of simple random walk on $G$ by $P$. Denote the eigenvalues of $P$ by $\lambda_n \le \cdots \le \lambda_2 < \lambda_1=1$. We say that $G$ is a \textbf{\emph{Ramanujan graph}} if $|\lambda_i| \in [0, 2d^{-1} \sqrt{d-1} ] \cup \{1\} $, for all $ i \le n$.

Let $p,q$ be two distinct prime numbers congruent to 1 modulo 4 such that $q>\sqrt{p}$ and $q \equiv a^2  $ modulo $p$ for some integer $a$. Then there exists a $(p+1)$-regular Ramanujan Cayley graph $G_{p,q}$ of size $q(q^2-1)/2$ whose girth is of size at least $2\log_p q $ \cite{cf:Xpq}.
We fix $p=5$ and take an increasing sequence $(q_n)_{n \in \N}$ of such prime numbers and consider  $H_n=G_{5,q_n}=(V'_n,E'_n)$. Fix some vertex $o\in V'_n$. Note that up to  a distance $\log_5 q_n $ from $o$ the graph $H_n$ looks like a $6$-regular tree.

\medskip

We take some sequences  of integers $s_n,m_{n},b_{n} $ tending to infinity such that:
\begin{itemize}
\item $\log_5 q_n/4 \le s_n^2m_nb_{n} \le \log_5 q_n/2$.
\item
$e^{m_n/b_n}=\Theta(s_n^2b_n)$.
\end{itemize}
We think of $b_n$ as\ tending to infinity arbitrarily slowly (compared to $s_n$). As $e^{m_n/b_n}=\Theta(s_n^2b_n)$ we think also of $m_n/ \log s_n $ and $(\log q_n)/(s_n^{2}\log \log q_n) $ as tending to infinity arbitrarily slowly.

\begin{figure}[h]
\begin{center}
\includegraphics[height=5cm,width=5cm]{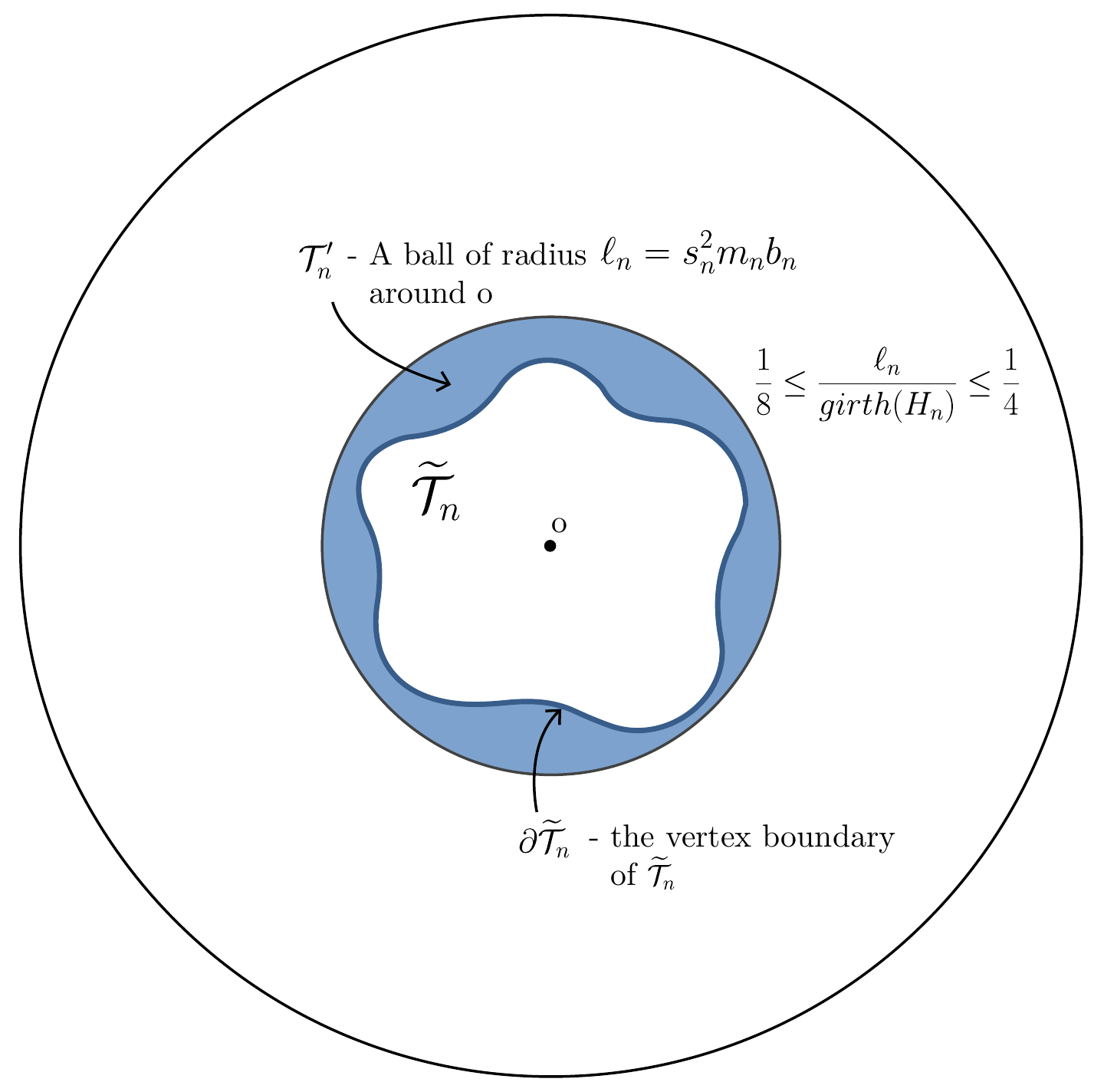}
\end{center}
\caption{\label{fig:4} This is a schematic representation of $H_n$ and the trees $\cT_n' $ and $\tilde \cT_n $}
\end{figure}
\begin{figure}[h]
\begin{center}
\includegraphics[height=5cm,width=5cm]{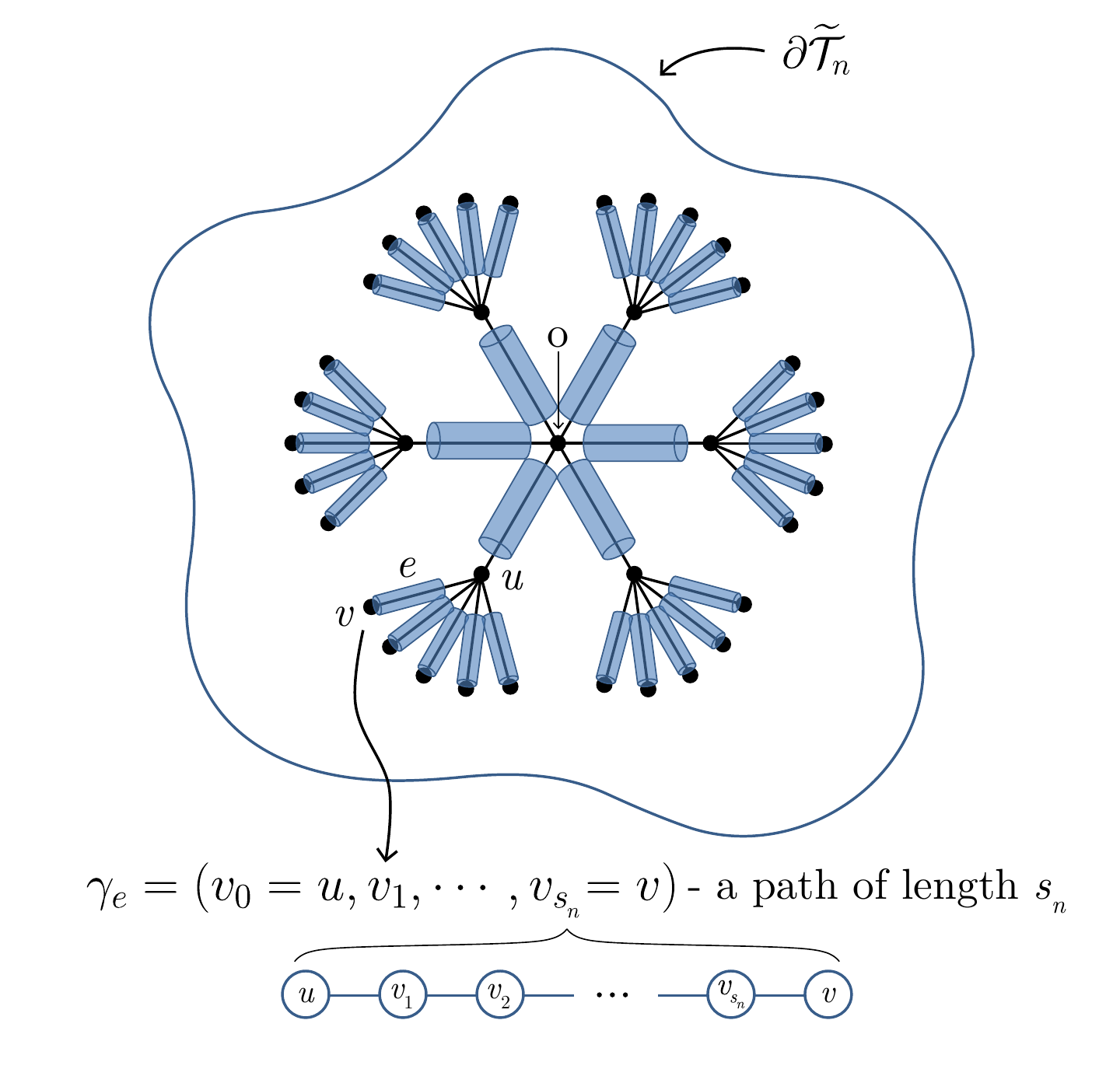}
\end{center}
\caption{\label{fig:5}The graph $G_n$ is obtained by stretching each edge of $\tilde \cT_n $ by a factor of $s_n$. Each cylinder above represents a path of length $s_n$, which replaces some edge of $\tilde \cT_n$.}
\end{figure}

 Denote the ball of radius $s_n^2m_nb_{n}$ around $o$ by $\cT'_n=(V(\cT_{n}'),E(\cT_n'))$. We think of $\cT'_n$ as a 6-regular tree rooted at $o$. We shall construct a sequence of graphs $G_n=(V_n,E_n)$ by stretching some of the edges of $H_n$ by a factor $s_n$ as follows. We shall pick a certain subtree $\widetilde \cT_n \subset \cT'_n$ (which is also rooted at $o$) and replace each edge $\{v,u \}$ in $\widetilde \cT_n$ by a path $\gamma_{u,v}$ of length $s_n$ whose end-points are $u$ and $v$. We call the resulting tree $\cT_n=(V(\cT_{n}),E(\cT_n))$. We identify each vertex of $\widetilde \cT_n=(V(\widetilde \cT_{n}),E(\widetilde \cT_n))$ with the corresponding vertex of $\cT_n$.

\medskip

The stretched edges have the effect of significantly ``slowing down" the walk while it is confined to $\cT_n$. With some care, we shall choose $\widetilde \cT_n$ so that
\begin{itemize}
\item $|\cT_n|/|V_n|=o(1/q_{n}^{2})$ and so the walk must escape $\cT_n$ before mixing.

\medskip

\item The distribution of the escape time from $\cT_n$ starting from the  root, is stochastically the largest (compared to all other starting points).

\medskip

\item The distribution of $T_{V(\cT_n)^c}/(3s_n^2m_nb_n)$, starting from $o$, is ``close" to the Exponential distribution with some constant mean (see \eqref{eq: geometric} for a precise statement).

\medskip

\item
Once the walk escapes $\cT_n$ it has a negligible chance of crossing any stretched edge by the time it is already extremely mixed. Consequently, there exists some $o(1)$ terms such that the additional amount of time required for the walk to become $\epsilon+o(1)$ mixed, beyond the time required for it to escape $\cT_n$ with probability of at least $1-\epsilon$, can be upper bounded by $t_{\mathrm{mix},H_n}(o(1)) \le C \log q_n =\Theta (s_n^2m_nb_n) $ (more precisely, we derive such a bound using Proposition \ref{prop: L2contractiononasubchain}). Putting all this together, it follows that
\begin{equation}
\label{eq: mixGneps}
\frac{t_{\mathrm{mix},G_{n}}(\eps)}{|\log \epsilon|}=\Theta(s_n^2m_nb_n)=\Theta(\log q_n), \text{uniformly for every }0<\epsilon \le 1/2.
\end{equation}
Hence there is no pre-cutoff, although the product condition holds (by Lemma \ref{lem: relaxation} $\rel(G_n)=O(s_n^2)=o(\log q_n)$).

\medskip

\item Under a certain $o(1)$-perturbation of some of the edges of $\cT_n $, w.h.p.~the walk would remain ``trapped" in $\cT_n$ until escaping it through the collection of its leaves which have maximal distance from the root. Consequently, as opposed to the situation in the original graph, the escape time from $\cT_n$ (starting from the root) is concentrated, but around a time of strictly larger order than $\log q_n $, namely $3s_n^4m_nb_{n}=\Theta(s_n^2 \log q_n)$. Thus the walk on the perturbed network exhibits a cutoff around time $3s_n^4m_nb_{n}$.
\end{itemize}

\medskip

We denote the internal vertex boundary of $\cT_n$ w.r.t.~$G_n$ and of $\widetilde \cT_n $ w.r.t.~$\cT_n'$  by $$\partial \cT_n:=\{v \in V(\cT_{n}) : \exists u \notin V(\cT_{n}) \text{ such that }\{u,v\}\in E_n \}, $$  $$\partial \widetilde \cT_n:=\{v \in V(\tilde \cT_{n}) : \exists u \notin V(\tilde \cT_{n}) \text{ such that }\{u,v\}\in E(\cT_n') \}.$$ By construction $\partial \widetilde \cT_n=\partial \cT_n$.
As $\widetilde \cT_n$ is rooted at $o$ and $\widetilde \cT_n \subset \cT_n'$, in order to define $\widetilde \cT_n$ it suffices to specify its collection of leaves, which is precisely the set $\partial \widetilde \cT_n$. Namely, $V(\tilde \cT_{n})$ is determined by $\partial \widetilde \cT_n$ as the union of the vertices along all paths in $\cT_n'$ from the root to $\partial \widetilde \cT_n$. We now describe our procedure for choosing the set $\partial \widetilde \cT_n $ (this concludes the construction).
\begin{figure}[h]
A 5-ary tree $\cT \cT $ of depth $m_n$. 
\begin{center}
\includegraphics[height=4cm,width=4cm]{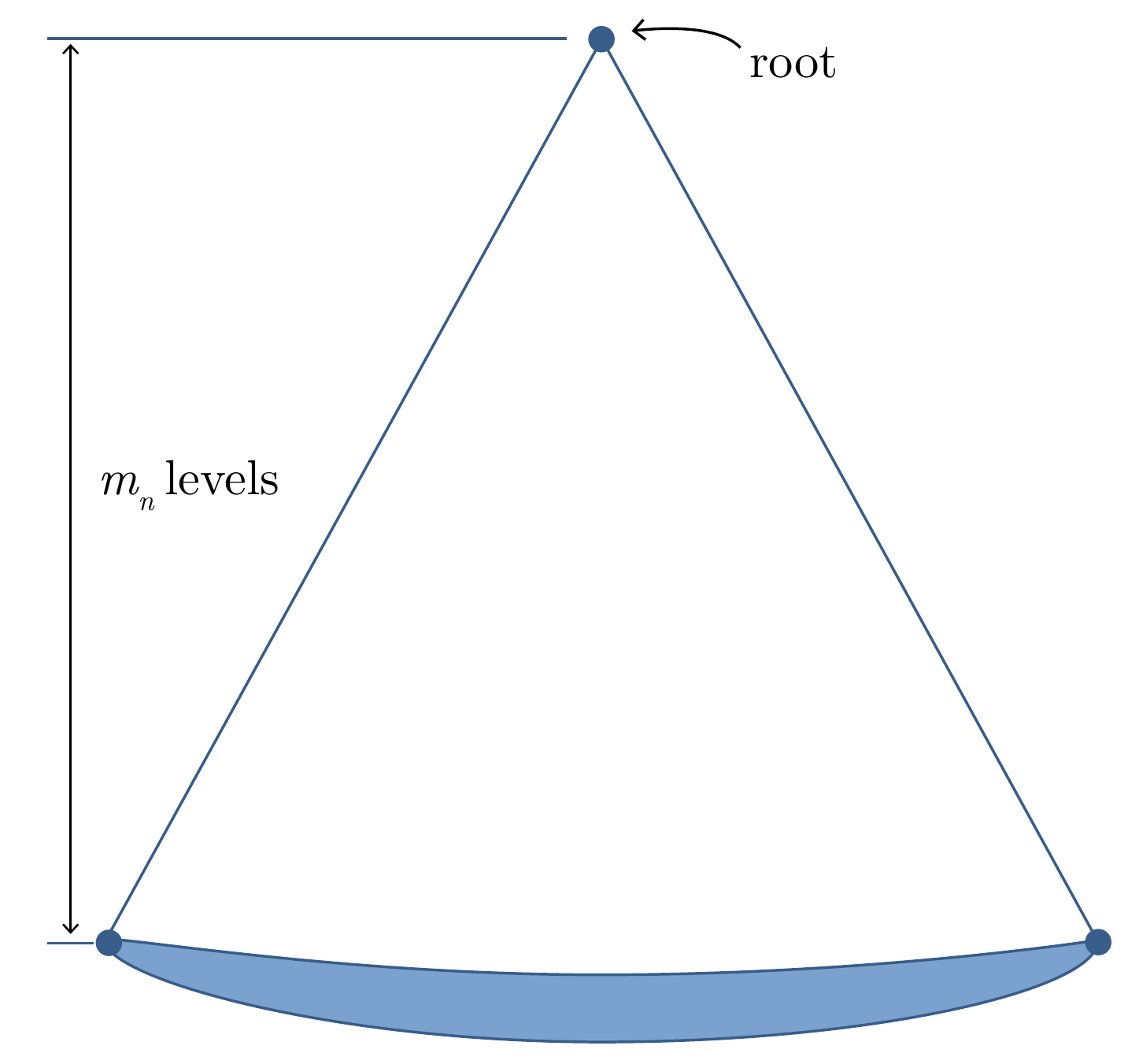}
\end{center}
\caption{\label{fig:TT}The leafs of $\cT \cT $ are partitioned into two parts, $A$ and $D$, where the set $D$ consist of leafs which are in some sense ``unbalanced". The set $D$ belongs to $\partial \tilde \cT_{n} $, while for every $a \in A$ another copy of $\cT \cT$ is contained in $ \tilde \cT_{n} $. We repeat this procedure for $s_{n}-1$ iterations (see Figure \ref{fig:7}), where in every iteration, the partition of the leaves into $A$ and $D$ is identical.}
\end{figure}

\begin{figure}[h]
A schematic representation of the recursive construction of $\tilde \cT_n $
\begin{center}
\includegraphics[height=8cm,width=9cm]{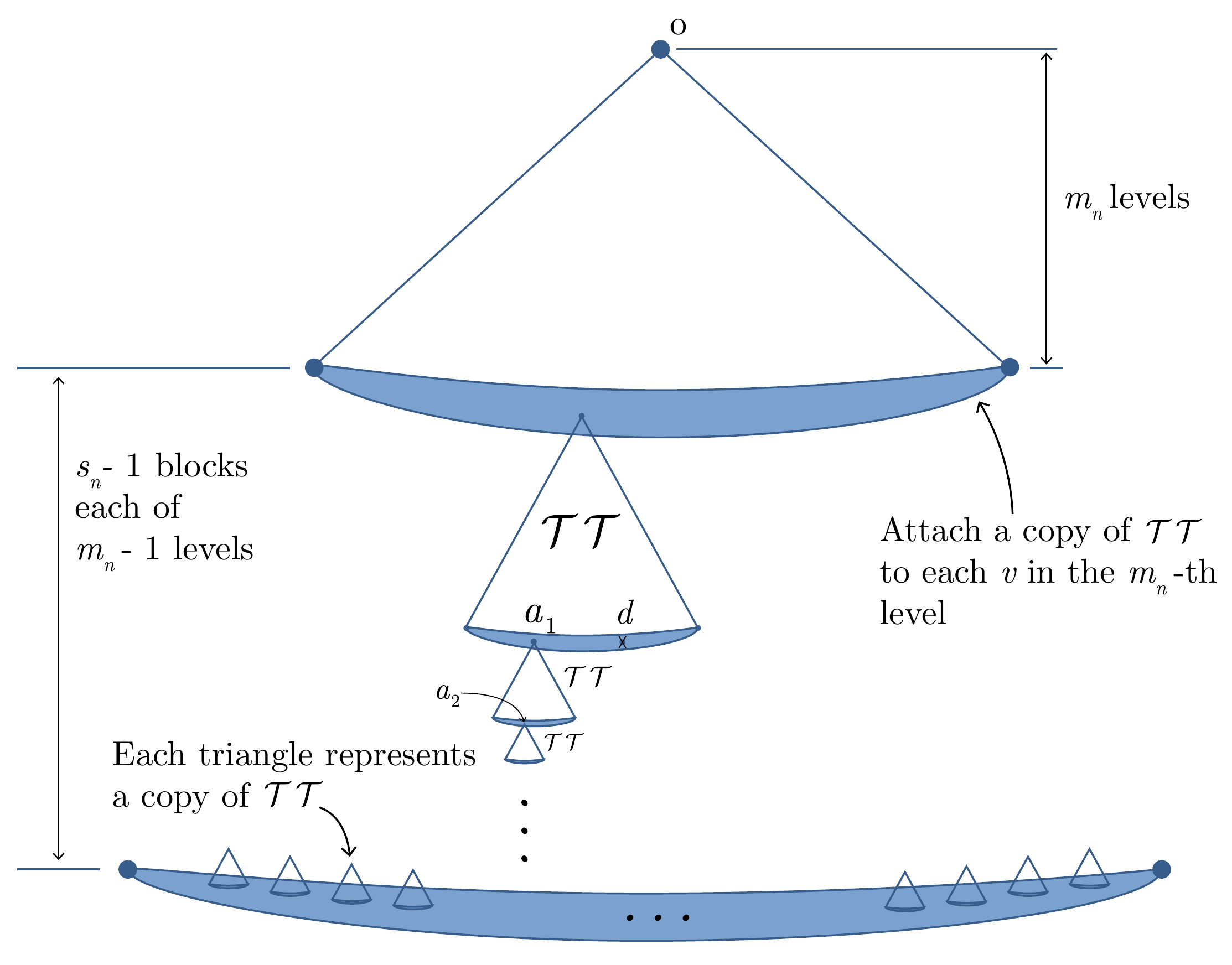}
\end{center}
\caption{\label{fig:7}Every triangle represents a copy of $\cT\cT$. Every $d \in D$ belongs to $\partial \tilde \cT$, while for every $a \in A$ another copy of $\cT \cT$ is contained in $ \tilde \cT_{n} $. We repeat this procedure for $s_{n}-1$ iterations, where in every iteration, the partition of the leafs into $A$ and $D$ is identical. The leafs of the copies of $\cT\cT$ from the last iteration all belong to  $\partial \tilde \cT$. Note that the first iteration is different, in that all of the leafs of the first copy of $\cT \cT$ (the one rooted at $o$) are in $A$.}
\end{figure}

\begin{itemize}
\item[(1)] Denote the $k$-th level of $ \cT'_n $  by $\cL_k(\cT'_{n})$ (this level contains the $ks_n$-th level of $\cT_n$). Then we shall construct $\partial \widetilde \cT_n$ so that  $\partial \widetilde \cT_n \subset \bigcup_{k=2}^{ s_n^2b_{n}}\cL_{km_n}(\cT_n')$.
We shall  define $D_{k}:=\partial \widetilde \cT_n \cap \cL_{km_n}(\cT_n') $ recursively starting from $k=2$ (i.e.~we set $\partial \widetilde \cT_n:=\bigcup_{k=2}^{ s_n^2b_{n}}D_k$). Having defined $D_{2}=\partial \cT_n \cap \cL_{2m_n}(\cT_n'),\ldots,D_{k}=\partial \cT_n \cap \cL_{km_n}(\cT_n') $, we define $A_{k+1}$ to be the set of vertices in $\cL_{(k+1)m_n}(\cT'_{n}) $ such that the path from them to $o$ in $\cT_n'$ does not go through any vertex in $\bigcup_{i=2}^{k}D_i$. The set $D_{k+1} $ shall be defined to be a certain subset of the vertices in  $\cL_{(k+1)m_n}(\cT'_{n}) $ which have a vertex in $A_k$ as an ancestor (this is described in (2)-(3) below). To start the construction we set $A_1:=\cL_{m_n}(\cT_n')$ and to conclude it we define $D_{s_n^2b_n}=\partial \cT_n \cap \cL_{s_n^2n_{b}m_n}(\cT_n')$ to equal $A_{s_n^2b_n}$ (making the last level of $\cT_n$ different). We now specify how $D_{k+1}$ is defined in terms of $A_k$ (first qualitatively (2) and then more concretely (3)).

\medskip

 \item[(2)]
  For every $k \le m_nb_{n}s_n^2-1 $ and every  $v \in \cL_{km_n}(\cT_n')$ we denote $$\cT_{v}:=\{u \in \bigcup_{i=0}^{m_n} \cL_{km_n+i}(\cT'_{n}): \text{ the path in }\cT'_n \text{ between }u \text{ and }o \text{ goes through }v \}.$$
Let $B_v$ be the the set of leaves of $\cT_v$. For all $2 \le k \le s_n^2b_n-1$ we will define $D_k$ to be a subset of $\bigcup_{v \in A_{k-1}}B_v$.
We shall define $\widetilde \cT_n$ so that for every $1 \le k_{1},k_2 \le s_n^2b_{n}-2 $ and every  $v_i \in A_{k_{i}}$ ($i=1,2$) the trivial isomorphism of $\cT_{v_1}$ and $\cT_{v_2}$ is a bijection from $B_{v_{1}} \cap \partial \widetilde \cT_n$ onto $B_{v_{2}}\cap \partial \widetilde \cT_n$. In particular,  $|B_v \cap \partial \widetilde \cT_n |/|B_v| $ is some fixed number, which we shall pick to be between $b_n^{-1}$ and $2b_n^{-1}$.

\medskip
\item[(3)]
We now define the sets $D_2,\ldots,D_{s_n^2b_n}$. For every vertex $v \in A_{k}$ for some $k<s_n^2b_n-1$ and every $u \in \cT_v \setminus B_v$, we distinguish one of the children of $u$ (w.r.t.~the tree $\cT_v$ viewed as a rooted tree with root $v$) as a \emph{left} child. For every $u \in B_v$ let $f(u)$ be the number of left children along the path from $u$ to $v$ in $\cT_v$. We define $F_v=B_v \cap \partial \widetilde \cT_n$ to be the collection of all $u \in B_v$ such that $f(u) \le m_n/5 - g_{n}(m_n)$, where $g_{n}$ (which may depend on $n$ and $(b_n,m_n)$) is chosen so that $1/b_n \le |B_v \cap \partial \cT_n |/|B_v| \le 2/b_n$ (for all sufficiently large $n$). Finally, we set $D_{k+1}:=\bigcup_{v \in A_k}F_v$.
\end{itemize}

\medskip

We start by describing four properties of $G_n$ which do not depend on the particular choice of  $\widetilde \cT_n$.
\begin{itemize}
\item[(a)] $\mathrm{girth}(G_n) \ge \mathrm{girth}(H_n) \ge 2\log_5 q_n$.

\medskip

\item[(b)] $|\cT_n | \le s_n |\cT_n' | \le s_n (5^{s_n^2m_nb_{n}}+1) \le s_n (5^{\frac{1}{2}\log_5 q_n}+1)=o(\sqrt{q_n \log q_n })$. Consequently, $|V_n|=q_{n}(q_{n}^2-1)/2+o(\sqrt{q_n \log q_n })$.

\medskip

\item[(c)]  $\rel(G_n)=O(s_n^2)$ (by Lemma \ref{lem: relaxation}). In particular, $\rel(G_n)=o(\log q_n)$ (which implies that the product condition holds).

\medskip

\item[(d)] There exist absolute constants $C_{1},C_{2}>0$ such that for every vertex $v \in V_n $ whose distance from $\cT_n$ is at least $C_1 \log \log q_n $ we have that $\|\Pr_v^{\lceil C_2 \log q_n \rceil} - \pi_n \|_{\TV}=o(1)$.

\medskip

\item[(e)] Let $C_1$ be as in (d). Denote by $J_n$ the collection of vertices whose distance from $\cT_n$ is at least $\lceil C_1 \log \log q_n  \rceil$. Then for every $u \in \partial \cT_n $ we have that $$\Pr_u[T_{J_{n}}>5C_1 \log \log q_n]=o(1).$$
\end{itemize}
(a) and (b) are trivial. For (d) note that for any $u$ in the exterior vertex boundary of $\cT_n$, the intersection of the ball of radius $\lfloor C_1 \log \log q_n \rfloor $ centered at $u$ with $V_n \setminus V(\cT_n) $ is a 5-ary tree.  Hence (d) follows from Proposition \ref{prop: L2contractiononasubchain}.

\medskip

For (e) observe that for any $v \in \partial \cT_n $ the probability that a lazy random walk started from $v$ reaches its parent $u$ w.r.t.~$\widetilde \cT_n $ (by crossing the path $\gamma_{u,v}$) before reaching $J_n$ is $o(1)$. Moreover,  the probability that the walk would spend at least $C_1 \log \log n$ steps in $\gamma_{u,v}$ before reaching $J_n$ is $o(1)$. Finally, conditioned on hitting $J_n$ before returning to $v$, the conditional distribution of $T_{J_{n}}$ is concentrated around $3C_1\log \log q_n $.

\medskip

\begin{itemize}
\item[(f)]  For any $t \ge 0$,
$$ \Pr_o[T_{\partial \cT_{n}}>t]=\max_{v \in \cT_n }\Pr_v[T_{\partial \cT_{n}}>t].$$
\item[(g)]
There exists an absolute constant $c_1>0$ such that for any
$2 \le k \le m_ns_n^2-1 $
\begin{equation}
\label{eq: geometric}
\begin{split}
& \Pr_o\left[T_{\partial \cT_n}>3s_n^2m_n (1+o(1))k \right] \le \left(1-b_n^{-1} \right)^{k-1}+o(1). \\ & \Pr_o\left[T_{\partial \cT_n}\le 3s_n^2m_n (1-o(1))k \right] \ge 1- \left(1-c_1b_n^{-1} \right)^{k-1}-o(1).
\end{split}
\end{equation}
\end{itemize}

Note that (\ref{eq: geometric}) together with (d)-(f) imply \eqref{eq: mixGneps}.

\medskip

It is easy to see how (f) follows from the symmetry of the construction together with the fact that the construction of the sets $D_i$ starts only from $k=2$.

We now explain (g). The function $f$ (from (3) above) could be extended to $V(\cT_n)$ by contracting the stretched edges into a single edge (where internal vertices of an edge are assigned the same value as one of the end-points of that stretched edge). Fix some $1 \le k \le s_n^2b_n-2$ and some $v \in A_k$. Let $Y$ be the last vertex in $\cL_{(k+1)m_n}(\widetilde \cT_n) $ visited by the walk prior to $T_{\partial \cT_n}$. Then starting from $v$, conditioned on $T_{\partial \cT_n}<T_{\cL_{km_n-1}(\widetilde \cT_n)}$ we have that the (conditional) law of $f(Y)$   w.r.t.~the walk on $\widetilde \cT_n$ is the same as its (conditional) law  w.r.t.~the walk on $\cT_n$. This law can be approximated well by that of a sum of $m_n$ i.i.d.~Bernoulli(1/5) r.v.'s. Similarly to \eqref{eq: hitD1}, the conditional probability that $X_{T_{\partial   \cT_n}} \in \cL_{(k+1)m_n}(\widetilde \cT_n)$ is at most $c^{-1}\Pr_v[Y \in \partial \cT_n \mid  T_{\partial \cT_n}<T_{\cL_{km_n-1}(\widetilde \cT_n)}]$.  Using this observation, it is easy to see that the probability that the walk reached $\cL_{km_n}(\tilde \cT_{n})=\cL_{ks_{n}^{2}m_n}(\cT_{n}) $ without first hitting $\partial \cT_n $ is between $(1-b_n^{-1} )^{k-1}$ and $(1-c_{1}b_n^{-1} )^{k-1} $. Whence (g) follows from the fact that starting from $o$ the hitting time of  $\cL_{km_n}(\cT'_{n}) $ (conditioned that it is hit before $\partial \cT_n$) is concentrated around $3s_n^2m_nk$.

To see this, first consider the hitting time of $\cL_{km_n}(\cT'_{n}) $ with respect to the non-lazy version of the induced chain on $\cT_n'$ (starting from $o$). Note that it is concentrated around $\frac{3}{2}m_nk $ (its distance from $o$ (w.r.t.~$\cT_n'$) is distributed like a biased nearest-neighbor random walk with a fixed bias). Finally, since for a lazy SRW on $\Z_+ $ we have that $\mathbb{E}_0[T_{\{-s_n,s_n \}}]=2s_n^2 $ and $\mathbb{E}_0[T_{\{-s_n,s_n \}}^{2}]=O(s_n^4) $, the conditional distribution of $T_{\cL_{km_n}(\cT'_{n})}$ (starting from $o$), given that $T_{\cL_{(k-1)m_n}(\cT'_{n})}>T_{\partial \cT_n}$, is extremely close to that of $\sum_{i=1}^{T}W_i$, where $T$ is distributed like $T_{\cL_{km_n}(\cT'_{n})} $ w.r.t.~the aforementioned non-lazy walk on $\cT_n'$ and $W_1,W_2,\ldots$ are i.i.d.~random variables of mean $2s_n^2$ and variance $O(s_n^4)$ (which are also independent of $T$). Hence (g) follows from the CLT in conjunction with the aforementioned concentration of $T$ around time $\frac{3}{2}m_nk$.

\medskip

We now describe the $o(1)$-perturbation described in the assertion of the theorem. For every vertex $v \in A_{k}$ for some $k<m_nb_{n}s_n^2-1$, we increase the edge weight of every edge belonging to some $\gamma_{u,w}$ such that $w$ is a left child of $u$ in $\cT_v$ to $1+1/b_{n}^{1/3}$.

\medskip

We consider also the perturbed network on $\widetilde \cT_n$ in which we increase the edge weight between each vertex and its left child to $1+1/b_n^{1/3}$.
Let $k,v$ and $Y$ be as in the paragraph following \eqref{eq: geometric}. A simple network reduction shows that  starting from $v$, conditioned on $T_{\partial \cT_n}<T_{\cL_{km_n-1}(\widetilde \cT_n)}$ we have that the (conditional) law of $f(Y)$  w.r.t.~the walk on the perturbed network on $\widetilde \cT_n$ is the same as its (conditional) law  w.r.t.~the perturbed walk on $\cT_n$. Similarly to part (2 a) of Fact \ref{fact: SRWonZ}, this law can be approximated well by that of a sum of $m_n$ i.i.d.~Bernoulli($1/5+O(1/b_n^{1/3})$) r.v.'s.

As before, up to constants, we may consider $\Pr_v[Y \in \partial \cT_n \mid  T_{\partial \cT_n}<T_{\cL_{km_n-1}(\widetilde \cT_n)}]$ rather than $\Pr_v[X_{T_{\partial   \cT_n}} \in \cL_{(k+1)m_n}(\widetilde \cT_n)\mid  T_{\partial \cT_n}<T_{\cL_{km_n-1}(\widetilde \cT_n)}]$ (here both probabilities are considered w.r.t.~the perturbed network). Note that the former event requires a deviation of order $m_n/b_n^{1/3}$ from the mean of $f(Y)$ and thus its conditional probability could be bounded from above by $\exp(-c'_2m_nb_n^{-2/3})$.

By the above discussion, it is easy to see that also after this perturbation (d), (e) and (f) remain valid and  that for all $k < s_n^2b_n $
$$\Pr_o\left[T_{\partial \cT_n} \ge T_{\cL_{(k+1)m_n}(\cT'_n)} \right] \ge 1-\left(1-\exp[-c_2m_nb_n^{-2/3}]\right)^{k}-o(1),$$
for some constant $c_2>0$. In particular, since we took $e^{m_n/b_n}=\Theta (s_n^2b_n)$ we get that $$\Pr_o\left[T_{\partial \cT_n} \neq T_{\cL_{s_n^2m_nb_n}(\cT'_{n})}  \right]=o(1).$$
As before, it follows that $T_{\partial \cT_n}$ is concentrated around $3s_n^4m_n b_n$. By (d) and (e) this implies that after the perturbation the random walk exhibits cutoff around this time. In particular, the order of the mixing time increased by a factor of $s_n^2$. Observe that from our assumption that $e^{m_n/b_n}=\Theta (s_n^2b_n)$ we get that  $\log |V_n|/\log \log |V_n| =O( (s_nb_n)^2)$. Note that by taking $b_n$ to tend to infinity arbitrarily slowly, we can increase the order of mixing time by a factor arbitrarily close to $\frac{\log |V_n|}{\log \log |V_n|}$ (as long as it is $o(\log |V_n|/\log \log |V_n|)$). \qed

\begin{remark}
\label{rem: Geo}
Lack of pre-cutoff implies that (along a certain subsequence) the chains mix ``very-gradually". We note that the graphs $G_n$ from Theorem \ref{thm sensitivity} exhibit, in some sense, the most gradual mixing possible.
  In general, $\mix (\eps) \le \lceil2 |\log_{2} \eps| \rceil \mix$, for all $0<\epsilon < 1/4$. Using the results in \cite{cf:cutoff}, it is not hard to show that for reversible chains, under the product condition, there is some $o(1)$ term (depending only on $\reln/\mixn$) such that \begin{equation*}
 \label{eq: gradual2}
\liminf_{n \to \infty}[\mixneps + \mixn(\delta -o(1))]/\mixn(\epsilon \delta) \ge 1, \text{ for all }0<\eps,\delta<1.
\end{equation*}  The graphs $G_n$ from Theorem \ref{thm sensitivity} satisfy that
$$  \mixneps \ge c |\log \eps| \mixn(1/2), \text{ for all }0<\eps<1/2 \text{ and all }n.$$

We note that (by stretching the stretched edges in the construction by a slightly larger factor) we could have constructed the graphs $G_n$ from Theorem \ref{thm sensitivity} so that $\mix(G_n)/\mathrm{girth}(G_n)$ tends to infinity arbitrarily slowly (but still $\mathrm{girth}(G_n)=\Theta (\mathrm{diameter}(G_n))$)  and so that
 \begin{equation}
 \label{eq: gradual}
\lim_{n \to \infty}[\mixneps + \mixn(\delta)]/\mixn(\epsilon \delta)=1, \text{ for all }0<\eps,\delta<1.
\end{equation}
The probabilistic interpretation of \eqref{eq: gradual} is that, loosely speaking, there is some random time $\tau_n$ (which can be taken to be a certain hitting time) having roughly a Geometric distribution such that the chain is extremely mixed in time $ \tau_n+o(\mathbb{E}[ \tau_n])$ (but its distance from $\pi$ is $1-o(1)$ at time $\tau_n-o(\mathbb{E}[ \tau_n])$).
\end{remark}

\section{Appendix}
\subsection{Hitting times connection to mixing times}
The aim of this section is to introduce some general theory which shall reduce the analysis of our examples to the analysis of hitting time distributions of certain sets.

\begin{definition}
\label{def: Linfty}
Let $(\Omega,P,\pi)$ be a finite  irreducible reversible Markov chain.  For
any $f \in \R^{\Omega}$,
let $\mathbb{E}_{\pi}[f]:=\sum_{x \in \Omega}\pi(x)f(x)$ and $\Var_{\pi}f:=\mathbb{E}_{\pi}[(f-\mathbb{E}_{\pi}f)^{2}]$.
The inner-product $\langle \cdot,\cdot \rangle_{\pi}$ and $L^{p} $ norm are
$$\langle f,g\rangle_{\pi}:=\mathbb{E}_{\pi}[fg] \text{ and } \|f \|_p:=\left(
\mathbb{E}_{\pi}[|f|^{p}]\right)^{1/p},\, 1 \le p < \infty,$$
and $\|f \|_{\infty}:=\max_{x \in \Omega }|f(x)|$. For any two distribution
$\mu  $ on $\Omega$ and $p \ge 1$ we define
$$\|\mu-\pi \|_{p,\pi}:=\|f_{\mu}-1\|_p, \quad \text{where }f_{\mu}(x):=\mu(x)/\pi(x). $$

\end{definition}

The following lemma is standard.
\begin{lemma}
\label{lem: contraction}
Let $(\Omega,P,\pi)$ be a finite lazy  irreducible reversible Markov chain.
Let
$\mu $ be a distribution on $\Omega$ and let $\lambda_2$ be the second largest
eigenvalue of $P$. Then for all $t \ge 0$
\begin{equation}
\label{eq: L2contraction}
2 \|\Pr_\mu^t-\pi \|_{\TV} \le \|\Pr_\mu^t-\pi \|_{2,\pi} \le \lambda_2^t
\|\mu-\pi \|_{2,\pi}.
\end{equation}

\end{lemma}
\begin{proof} The first inequality in (\ref{eq: L2contraction})
follows from the fact that $2 \|\Pr_\mu^t-\pi \|_{\TV}=\|\Pr_\mu^t-\pi \|_{1,\pi} \le \|\Pr_\mu^t-\pi \|_{2,\pi}   $ by the Cauchy-Schwatrz inequality. The second inequality
in (\ref{eq: L2contraction}) is
proved using the spectral decomposition in a straightforward manner (e.g.~\cite[Lemma 3.26]{cf:Aldous}).
\end{proof}

\begin{definition}
\label{def: Cheeger}
Let $(\Omega,P,\pi)$ be a finite Markov chain. For a set $A \subset \Omega $ denote $Q(A):=\sum_{x \in A,y \notin A }\pi(x)P(x,y) $. Denote $\Phi(A):=Q(A)/\pi(A) $. We define the \textbf{{\em Cheeger constant}} of the chain as $\Phi:=\min_{A:0< \pi(A) \le 1/2}\Phi(A) $.
\end{definition}
\begin{definition}
Let $G=(V,E)$ be a finite graph. The edge boundary of a set $S \subset V $ is defined as $\partial_{E}S:=\{\{s,s'\} \in E :s \in S,s' \notin S \}$.  The Cheeger constant of lazy simple random walk on $G$ is defined as
$$\mathrm{ch}_{\mathrm{L}}(G):=\min_{S:0<\pi(S) \le 1/2}\frac{|\partial_{E}S|}{2\sum_{v \in S}\deg (v)}, $$
which coincides with Definition \ref{def: Cheeger} (see e.g.~\cite[Remark 7.2]{cf:LPW}).
We say that $G$ is a $c$-lazy expander if $\mathrm{ch}_{\mathrm{L}}(G)>c $. We say that a sequence of finite graphs $(G_n)_{n \ge 1}$ is a family of $c$-lazy expanders if $\inf_n \mathrm{ch}_{\mathrm{L}}(G_{n})>c
$.
\end{definition}

\medskip

The following theorem is the well-known discrete analog of Cheeger's inequality \cite{cf:Alon1,cf:Alon2,cf:Sinclair} (the proof could also be found at \cite[Theorem
13.14]{cf:LPW}).
\begin{theorem}
\label{thm: Cheeger}
Let $\lambda_2$ be the second largest eigenvalue of a reversible transition matrix on a finite state space. Let $\Phi$ be as in Definition \ref{def: Cheeger}.  Then
\begin{equation}
\label{eq: Sinclair}
\Phi^2/2 \le 1- \lambda_2 \le 2\Phi.
\end{equation}
\end{theorem}
The following proposition enables  us to reduce the problem of bounding $d(t)$ from above  in the proofs of Theorems \ref{thm sensitivity} and \ref{thm: sensitive2} to the problem of estimating the probability that a certain large set $A$ (``the center of mass of the chain") was not hit by time $t$.
\begin{proposition}
\label{prop: L2contractiononasubchain}
Let
$G=(V,E)$ be a finite connected graph. Fix some edge weights $(c_e)_{e \in E}$. Assume that $1 \le \sum_{u}c_{v,u} \le D$ for all $v \in V$. Let $(V,P,\pi)$ be the  lazy random walk on the corresponding network.

 Let $0<\epsilon < 1$. Let $A \subset V$. Let $\partial_{V} A:=\{a
\in A:\exists b \notin A \text{ s.t. }\{a,b\}\in E  \}$ and $\tilde A:=A \setminus
\partial_{V} A$. Assume that $\pi(\tilde A) \ge 1-\eps/3$. Let $(A,\tilde P,\tilde \pi)$ be   lazy
random walk on the induced network on $A$. Denote its Cheeger constant by $c$.  Denote $r:=\left\lceil \frac{2}{c^2} \log
\left(\frac{3 D |A| }{2\epsilon} \right) \right\rceil $.
Let $x \in \tilde A $ be such that $\Pr_{x}[T_{\partial_V A} < r] \le \eps/3
$. Then
$$\|\Pr_x^{r}-\pi \|_{\TV} \le \eps.$$
\end{proposition}
\begin{proof}
First note that by (\ref{eq: L2contraction}) together with (\ref{thm: Cheeger})  $$\|\tilde P^r(x,\cdot)-\tilde \pi\|_{\TV}\le \frac{1}{2} \|\tilde P^r(x,\cdot)-\tilde \pi\|_{2,\tilde \pi} \le \frac{1}{2\tilde \pi(x)}\lambda_2^{r} \le \frac{D|A|}{2}e^{-\log \left(\frac{3D|A|}{2\epsilon} \right)} =\eps/3.$$ We also have that  $\|\tilde \pi-\pi \|_{\TV} \le 1-\pi(\tilde A) \le \eps/3$. By a straightforward
coupling argument $\|\Pr_x^{r}(\cdot)-\tilde P^{r}(x,\cdot)
\|_{\TV} \le \Pr_{x}[T_{\partial A} < r] \le \eps/3  $.  Finally, by the triangle inequality we get that $\|\Pr_x^{r}-\pi
\|_{\TV} \le 3 \cdot \frac{\eps}{3}=\eps$.
\end{proof}
\subsection{A useful lemma for bounding the relaxation-time}
The following lemma allows us to easily bound the relaxation-time of our examples.

\begin{lemma}
\label{lem: relaxation}
Let $G_n=(V_n,E_n)$ be a family of $c$-lazy expanders.
\begin{itemize}
\item[(i)] Let $H_n$ be a sequence of graphs obtained by stretching some of the edges in $G_n$ by a factor of $s_n$. Then $\rel(H_n)=O(s_n^2)$.
\item[(ii)] Let $F_n=(W_n,U_n)$ be a sequence of graphs obtained by decorating some of the vertices of $G_n$ with a 3 dimensional torus of side length $k_n$. Then $\rel(H_n)=\Omega(k_n^3)$.
\end{itemize}
\end{lemma}
\begin{proof}
 Since $G_n$ is an expander and $H_n$ is obtained from it by stretching some of the edges by a factor of $s_n$, the Cheeger constant of $H_n$ is at least of order $s_n^{-1}$ (e.g.~\cite[Proposition 2.3]{Hermon2016}). Hence $\rel(H_n)=O(s_n^2)$ by (\ref{eq: Sinclair}). Part (ii) is obtained from \eqref{eq: Sinclair}.
\end{proof}
\subsection{Proof of Fact \ref{fact: SRWonZ}}
\label{s:F41}
The proof of part (1 a) is easy and hence omitted. We now prove (1 b). The upper bounds in \eqref{eq: hitD1} are trivial. We now prove that $ \Pr_o[X_{\tau_n} \in D],\Pr_o[X_{T_{\cL_n}} \in D]  \ge c \Pr_o[\mathcal{D}]$ (for some $0<c<1$, independent of $D$ and $n$). This follows from the fact that for all $d \in D$, the probability that $D$ was visited and that $d$ is the last (resp.~first) vertex in $D$ to be visited is at least $c\Pr_o[X_{\tau_n} =d] $ (resp.~$c\Pr_o[X_{T_{\cL_n}} =d] $). We leave the details to the reader. The proof of (2 b) is analogous and hence omitted.

We now prove (2 a). Denote the left and right children of $o$ by $u$ and $v$, respectively. Let $\cT_v$ and $\cT_u$ be the subtrees rooted at $v$ and $u$, resp.. We define the left and right trees rooted at $o$ to be the trees obtained by deleting $\cT_v$ and $\cT_u$, resp.. Write $w$, $w_L$ and $w_R$ for the conductance from the root to infinity in the original tree, the left tree and the right tree, respectively. Let $A$ be the event that the last vertex of $\{v,u\}$ which was visited by the walk was $u$ (i.e.~the walk got absorbed in $\cT_u$). Using a standard network reduction and the fact that  $\cT_v$ and $\cT_u$ are identical to $\cT$, we get the following relations: $w=w_{L}+w_R$, $\frac{1}{w_L}=\frac{1}{1+\epsilon}+\frac{1}{w} $,   $\frac{1}{w_R}=1+\frac{1}{w} $ and $\Pr_o[A]=\frac{w_L}{w_L+w_R}$. Solving this system of equations yields that $\Pr_o[A]= \frac{\sqrt{1+\epsilon}}{1+\sqrt{1+\epsilon}}$. Using again the fact that  $\cT_v$ and $\cT_u$ are identical to $\cT$, allows us to repeat the argument and the claim now follows by induction. \qed

\subsection{Proof of Lemma \ref{lem: pathsdecomposition}}
\label{s:lem32} 
We first prove \eqref{eq: septhroughz1}. We only prove the discrete-time lazy case, as the continuous-time case is analogous. By reversibility and the Markov property w.r.t. $\min (T_z,T_y) $
\begin{equation}
\label{eq: TAxy1}
\begin{split}
& \frac{P_{\mathrm{L}}^t(x,y)}{\pi(y)}=\sum_{k_{1}=0}^t \frac{\Pr_x[T_{z}=k_{1}<T_{y}]P_{L}^{t-k_{1}}(z,y)+}{\pi(y)}+
\frac{\Pr_x[T_{y}=k_{1}<T_{z}]\Pr_{L}^{t-k_{1}}(y,y)}{\pi(y)} \\ &
= \sum_{k_{1} =0  }^t \frac{\Pr_x[T_{z}=k_{1}<T_{y}]P_{\mathrm{L}}^{t-k}(y,z)}{\pi(z)}+ \sum_{k =0  }^t \mathbb{P}[T_{z,y}^{x}=k,T_y^x \le T_z^x ]P_{\mathrm{L}}^{t-k}(y,y)/ \pi(y)
\end{split}
\end{equation}
Now $ \sum_{k _{1}\le t } \frac{\Pr_x[T_{z}=k_{1}<T_{y}]P_{\mathrm{L}}^{t-k}(y,z)}{\pi(z)}= \sum_{ k_1,k_2 :k_1+k_2 \le t } \frac{\Pr_x[T_{z}=k_{1}<T_{y}]\Pr_y[T_{z}=k_{2}]P_{\mathrm{L}}^{t-k_{1}-k_2}(z,z)}{\pi(z)}$, which equals $  \sum_{k:k \le t }\mathbb{P}[T_{z,y}^{x}=k,T_y^x > T_z^x ]P_{\mathrm{L}}^{t-k}(z,z)/ \pi(z) $. Substituting this in \eqref{eq: TAxy1} yields the equality in \eqref{eq: septhroughz1}. The inequality follows from the fact that for all $a \in \Omega$ we have that $P_{\mathrm{L}}^{s}(a,a)/\pi(a) $ is decreasing in $s$ and tends to 1 as $s \to \infty$ which follows from the spectral decomposition and the non-negativity of the eigenvalues of $P_{\mathrm{L}}=\frac{1}{2}(I+P)$. For \eqref{eq: septhroughz3} use \eqref{eq: septhroughz1} with $z=y$. We now prove \eqref{eq: septhroughz2}. Here we only prove the continuous-time case.
 
Denote the densities of $T_z$ under $\h_x$ and $\h_y$ by $f_z^x$ and $f_z^y$, resp.. Conditioning on $T_z$ (which is deterministically smaller than $T_y$ under $\h_x$ in our current setup), then using reversibility and finally, conditioning on $T_z$ again (now under $\h_y$), we get that \[\frac{H_t(x,y)}{\pi(y)}= \int_{0}^{t}f_z^{x}(s)\frac{H_{t-s}(z,y)ds}{  \pi(y)}= \int_{0}^{t}f_z^{x}(s)\frac{H_{t-s}(y,z)ds}{ \pi(z)}\] \[ = \int_{s,r \ge 0:s+r \le t}f_z^{x}(s)f_{z}^y(r)\frac{H_{t-(s+r)}(z,z)drds}{ \pi(z)}=\int_{0}^{t}f_z^{x,y}(s)\frac{H_{t-s}(z,z)ds}{ \pi(z)} \] \[= \mathbb{P}[\tau_{z}^{x,y} \le t]+\int_{0}^{t}f_z^{x,y}(s)\frac{H_{t-s}(z,z)-\pi(z)}{ \pi(z)} ds. \]
Denote $M:=\max_sf_z^{x,y}(s)$. Denote the minimal non-zero eigenvalue of $I-P$ by  $\lambda $. Using the spectral decomposition $\frac{H_{s}(z,z)-\pi(z) }{ \pi(z)} \le e^{-\lambda s}\frac{H_{0}(z,z)-\pi(z) }{ \pi(z)}=\frac{1-\pi(z) }{ \pi(z)}e^{-\lambda s}  $, for all $s$, thus \[ \int_{0}^{t}f_z^{x,y}(s)\frac{H_{t-s}(z,z)-\pi(z) ds}{ \pi(z)} \le M \int_{0}^{\infty}\frac{H_{s}(z,z)-\pi(z) ds}{ \pi(z)} \le M\frac{1-\pi(z) }{ \pi(z)}\rel.  \]
Substituting this above gives \eqref{eq: septhroughz2}. For \eqref{eq: septhroughz4} use \eqref{eq: septhroughz2} with $z=y$.
\qed

\subsection{Proof of \eqref{eq: tdel}:}
\label{s:ap64} 
Since $H_t=e^{t(P-I)}=e^{2t(P_L-I)}$, the continuous-time version of the chain is also the continuous-time version of the lazy chain, but run twice faster. Thus, in order to sample $\tau_z^{a,b}$, we may first sample $T_{z}^{a,b} $ and a Poisson process on $\R_+$ with rate $2$, $(N_t)_{t \ge 0 }$, and then set $\tau_z^{a,b}=\inf \{t:N_t \ge T_{z}^{a,b} \} $. Thus for all $s,t$ we have
\begin{equation}
\label{eq:st}
\mathbb{P}[\tau_z^{a,b} \le t] \ge \mathbb{P}[N_t \ge s ]\mathbb{P}[T_{z}^{a,b} \le s].
\end{equation} 
Using \eqref{eq:st} (for the first inequality in \eqref{eq:LDTz0}), we argue that there exists some sufficiently small absolute constant $c>0$ such that for all $0<\delta<1/2$ we have for all sufficiently large $n$ that 
\begin{equation}
\label{eq:LDTz0}
\mathbb{P}[\tau_z^{a,b} \le \frac{1}{2}(t_{\del}-cn \sqrt{\del})] \ge \mathbb{P}[\mathrm{Pois}(t_{\del}-cn \sqrt{\del}) \ge t_{\del}+cn \sqrt{\del}]\mathbb{P}[T_{z}^{a,b} \le t_{\del}+cn \sqrt{\del}] \ge 2^{-\delta n}.
\end{equation}
In other words, (when $c$ is sufficiently small) there is an entropy gain in allowing $T_z^{a,b}$ to perform a smaller large deviation (requiring $T_{z}^{a,b} \le t_{\del}+cn \sqrt{\del} $ instead of $T_{z}^{a,b} \le t_{\del}$) while forcing a (relatively) small large deviation in the  number of steps performed by the continuous-time chain (requiring that by time $t_{\del}-cn \sqrt{\del}$ it makes at least $t_{\del}+cn \sqrt{\del}$ steps). This behavior is typical in the theory of large deviations. We now justify this claim.    

The large deviation behavior of $T_{z}^{a,b} $ is studied in detail in \cite[Example 3 and Lemma 5.1]{cf:Hermon}. In particular, it is shown that it has the same rate function as that of a certain sum i.i.d.~r.v.'s with exponential tails (see also \S~\ref{sec: lazypq}). Using a second order Taylor expansion for the rate function around the value corresponding to the mean (and the fact that the first derivative of the rate function vanishes at that point), it is also shown that for all $a \in (0,1/2] $  we have that \begin{equation}
 \label{eq: LDTz1}
c_1 a^2 \le - \frac{1}{n} \log \mathbb{P}[T_{z}^{a,b} \le 12n - an] \le C_1 a^2.
\end{equation}
Using convexity of the rate function (the Legendre transform of a convex function is itself convex)  for all $a \in (0,1/2] $  and $0<c< a $ we have that \[-\frac{1}{ n} \log \mathbb{P}[T_{z}^{a,b} \le 12n-(a+c)n] \ge - \frac{a+c-o(1)}{a} \frac{1}{ n} \log \mathbb{P}[T_{z}^{a,b} \le 12n-an] , \]
and so by \eqref{eq: LDTz1}
\begin{equation}
 \label{eq: LDTz2}
\frac{1}{ n} \log \left( \frac{  \mathbb{P}[T_{z}^{a,b} \le 12n -an]}{ \mathbb{P}[T_{z}^{a,b} \le 12n-(a+c)n] } \right) \ge -\frac{c-o(1)}{a} \frac{1}{ n}  \log \mathbb{P}[T_{z}^{a,b} \le 12n-an]  \ge c_{1}ac-o(1). 
\end{equation}
Using \eqref{eq: LDTz2}, it is not hard to see that indeed    
\begin{equation}
 \label{eq: LDTz3}
c_0\sqrt{\del}n \le 12n-t_{\del} \le C_0\sqrt{\del}n. 
\end{equation}
Using the fact that $\Pr[\mathrm{Pois}(\mu)>\mu(1+\epsilon)] \ge \exp \left(-C_{2}\epsilon^{2} \mu \right)$ for some absolute constant $C_2$ for all $\mu>0$ and $0<\epsilon<1$ yields that for some $\epsilon'>0$, for all
 $0<\delta \le 1/2 $ and $0<c<\epsilon' $ we have \[ \mathbb{P}[\mathrm{Pois}(t_{\del}-cn \sqrt{\del}) \ge t_{\del}+cn \sqrt{\del}] \ge e^{-C_{2}c^{2}\delta n}. \]
By \eqref{eq: LDTz2}-\eqref{eq: LDTz3} and the definition of $t_{\delta}$ we have  that for such $c$ and $\delta$
\[\mathbb{P}[T_{z}^{a,b} \le t_{\del}+cn \sqrt{\del}] \ge \mathbb{P}[T_{z}^{a,b} \le t_{\del}]e^{c_{2}c\delta n} \ge  2^{-\delta n}e^{c_{2}c\delta n}.  \] The last two inequalities imply \eqref{eq:LDTz0} with $c=c_2/(2C_2)$.
 \qed

\end{document}